\documentclass[oneside,reqno]{amsart}
\usepackage[utf8]{inputenc}
\usepackage{xcolor}
\usepackage{url}
\usepackage{enumitem}
\usepackage{amstext}
\usepackage{amsthm}
\usepackage{amssymb}
\usepackage{esint}
\PassOptionsToPackage{normalem}{ulem}
\usepackage{ulem}
\usepackage[unicode=true,pdfusetitle,
 bookmarks=true,bookmarksnumbered=false,bookmarksopen=false,
 breaklinks=false,pdfborder={0 0 0},pdfborderstyle={},backref=false,colorlinks=true]
 {hyperref}
\hypersetup{
 pdfborderstyle={},pdfborderstyle={}}

\makeatletter

\providecolor{lyxadded}{rgb}{0,0,1}
\providecolor{lyxdeleted}{rgb}{1,0,0}
\DeclareRobustCommand{\lyxsout}[1]{\ifx\\#1\else\sout{#1}\fi}

\numberwithin{equation}{section}
\numberwithin{figure}{section}
\theoremstyle{plain}
\newtheorem{thm}{\protect\theoremname}
\theoremstyle{plain}
\newtheorem{lem}[thm]{\protect\lemmaname}
\theoremstyle{definition}
\newtheorem{defn}[thm]{\protect\definitionname}
\theoremstyle{plain}
\newtheorem{cor}[thm]{\protect\corollaryname}
\theoremstyle{remark}
\newtheorem{rem}[thm]{\protect\remarkname}
\theoremstyle{definition}
\newtheorem{example}[thm]{\protect\examplename}

\@ifundefined{date}{}{\date{}}
\usepackage{pdfsync}
\usepackage{graphics}
\usepackage{epstopdf}
\usepackage[all]{xy}
\usepackage{amscd}
\usepackage{enumitem}
\usepackage{mathtools}
\usepackage[utf8]{inputenc}
\setlist[enumerate]{leftmargin=*,label=(\roman*),align=left}

\makeatletter
\@namedef{subjclassname@2020}{%
  \textup{2020} Mathematics Subject Classification}
\makeatother
                         
\newcommand{\ra}{\longrightarrow}
\newcommand{\field}[1]{\mathbb{#1}}
\newcommand{\R}{\field{R}} % reals
\newcommand{\N}{\field{N}} % naturals
\newcommand{\mC}{\mathcal{C}}
\DeclareMathOperator{\LoD}{C}

\newcommand{\eps}{\varepsilon} % for the sake of brevity only
\renewcommand{\phi}{\varphi} 
\newcommand{\diff}[1]{\,\hbox{\rm d}#1} % dt,dx,... for integrals
\newcommand{\Coo}{\mathcal{C}^{\infty}}

\newcommand{\ptind}{\displaystyle \mathop {\,\ldots\,}} % marks with smth over. Usage: \ptind^{...}
\newcommand{\ptindShort}{\displaystyle \mathop {\,\ldots\,}} 
\newcommand{\frontRise}[2]{\ifmmode\mathchoice{{\vphantom{#1}}^{\scalebox{0.6}{$#2$}}}  % displaystyle
 {{\vphantom{#1}}^{\scalebox{0.56}{$#2$}}}  % normal 
 {{\vphantom{#1}}^{\scalebox{0.47}{$#2$}}}  % scriptstyle 
 {{\vphantom{#1}}^{\scalebox{0.35}{$#2$}}}\fi} % scriptscriptstyle 
\newcommand{\rti}{\frontRise{\R}{\rho}\widetilde \R}

\makeatother

\providecommand{\corollaryname}{Corollary}
\providecommand{\definitionname}{Definition}
\providecommand{\examplename}{Example}
\providecommand{\lemmaname}{Lemma}
\providecommand{\remarkname}{Remark}
\providecommand{\theoremname}{Theorem}

\begin{document}

\title{A Picard-Lindelöf theorem for smooth PDE}
\author{Paolo Giordano \and Lorenzo Luperi Baglini}
\thanks{P.~Giordano has been supported by grants P30407, P33538 and P34113
of the Austrian Science Fund FWF}
\address{\textsc{Faculty of Mathematics, University of Vienna, Austria, Oskar-Morgenstern-Platz
1, 1090 Wien, Austria}}
\thanks{L.~Luperi Baglini has been supported by grant P30821 of the Austrian
Science Fund FWF}
\address{\textsc{Faculty of Mathematics, Università di Milano, Italy, Via Cesare
Saldini 50, 20133 Milano, Italy}}
\email{\texttt{paolo.giordano@univie.ac.at}, \texttt{lorenzo.luperi@unimi.it}}
\thanks{The authors thank A.~Bryzgalov for the idea of Example \ref{exa:Aleksandr},
and I.~Ekeland, M.~Kunzinger and É.~Séré, who were the first to
check this paper.}
\keywords{to do}
\subjclass[2020]{35GXX, 47J07, 46-XX}
\begin{abstract}
We prove that Picard-Lindelöf iterations for an arbitrary smooth normal
Cauchy problem for PDE converge if we assume a suitable Weissinger-like
sufficient condition. This condition includes both a large class of
non-analytic PDE or initial conditions, and more classical real analytic
functions. The proof is based on a Banach fixed point theorem for
contractions with loss of derivatives. From the latter, we also prove
an inverse function theorem for locally Lipschitz maps with loss of
derivatives in arbitrary graded Fréchet spaces.
\end{abstract}

\maketitle

\section{\label{sec:intro}Introduction}

Starting from the work of H.~Lewy \cite{Lew57}, it is clear that
a general Picard-Lindelöf theorem (PLT) for Cauchy problems of the
form:

\begin{equation}
\begin{cases}
\partial_{t}^{d}y(t,x)=F\left[t,x,\left(\partial_{x}^{\alpha}\partial_{t}^{\gamma}y\right)_{\substack{|\alpha|\leq L\\
\gamma\le p
}
}\right],\\
\partial_{t}^{j}y(t_{0},x)=y_{0j}(x)\ j=0,\ldots,d-1,
\end{cases}\label{eq:PDE}
\end{equation}

\noindent is not possible (see also e.g.~\cite{Den06} and references
therein for the more general problem of solvability of partial differential
operators). In \eqref{eq:PDE}, we consider $y$, $y_{0}^{j}$, $F$
as arbitrary ($\R^{m}$-valued) smooth functions, $(t,x)\in T\times S\subseteq\R\times\R^{s}$,
$\alpha\in\N^{s}$, $\gamma\in\N$, $p$, $L\in\N$, $d\in\N_{>0}$,
and we assume that $p<d$. In this paper, we show the convergence
of Picard-Lindelöf iterations of the general problem \eqref{eq:PDE}
under a suitable sufficient condition depending both on the initial
conditions $y_{0}^{j}$ and the function $F$. We also prove that
this condition includes non-trivial cases where $F$ could be non-analytic,
and a large class of smooth non-analytic initial conditions $y_{0j}$.

According to \cite{Eke11,EkSe}, one of the main problems in trying
to solve \eqref{eq:PDE} using Picard-Lindelöf iterations is that
the corresponding fixed point integral operator $P$ has \emph{$L\in\N$
loss of derivatives,} i.e.~satisfies $\left\Vert P^{n+1}\left(y_{0}\right)-P^{n}\left(y_{0}\right)\right\Vert _{k}\leq\alpha_{kn}\left\Vert P\left(y_{0}\right)-y_{0}\right\Vert _{k+nL}$
for all $k$, $n\in\N$ (here we are using the notion of ``loss of
derivatives'' as in \cite{Mos61,Eke11,EkSe}, and not as e.g.~in
\cite{PaPa05,Koh05,Koh13}; see Def.~\ref{def:contractionLoD} below
for a formal definition). For this reason, in Sec.~\ref{sec:BFPTlossDer},
we first generalize the Banach fixed point theorem (BFPT) to contractions
with loss of derivatives, and we will see that the aforementioned
sufficient condition corresponds to a Weissinger-like assumption,
\cite{Wei52}. In Sec.~\ref{sec:equationsLoD}, we hence apply this
BFPT to prove an inverse function theorem in arbitrary graded Fréchet
spaces (not necessarily of tame type or with smoothing operators,
like in Nash-Moser theorem, see \cite{Mos61,Ham82}) and for locally
Lipschitz maps with loss of derivatives (non necessarily differentiable
maps, like in Ekeland inverse function theorem, see \cite{Eke11}).
In Sec.~\ref{sec:PLTpde}, this BFPT with loss of derivatives is
used to prove a PLT for normal PDE. In Sec.~\ref{sec:Examples},
we apply this PLT to a family of PDE including both a non-analytic
$F$ or non-analytic initial conditions. Finally, in Sec.~\ref{sec:aboutLOD},
we present a preliminary study of the notion of contraction with loss
of derivatives.

In the following, we say that the Cauchy problem \eqref{eq:PDE} is
in \emph{normal form} to specify that the highest derivative in $t$
(called \emph{normal variable}) can be isolated on the left hand side
of the PDE (some authors call this problem in Kowalevskian form).

If $y:X\ra\R^{m}$, then $y^{h}:X\ra\R$ is the $h=1,\ldots,m$ component
of $y$, and in $\N=\{0,1,2,\ldots\}$ we always include zero. Therefore,
the notations $\left(\partial_{x}^{\alpha}\partial_{t}^{\gamma}y\right)(t,x)$
used in \eqref{eq:PDE} include cases where some $\alpha_{j}=0$,
$j=1,\ldots,s$, or $\gamma=0$. Finally, $\mC^{k}(X,\R^{m})$ denotes
the set of all the $\mC^{k}$ functions $f:X\ra\R^{m}$, whereas $\mC^{k}(X):=\mC^{k}(X,\R)$.

\section{\label{sec:BFPTlossDer}A Banach fixed point theorem with loss of
derivatives}

The idea to extend the classical Banach fixed point theorem to sequentially
complete subsets $X$ of Hausdorff locally convex linear spaces $(E,(|-|)_{\alpha\in\Lambda})$
dates back to \cite{CaNa71}. Here, a contraction is a map $P:X\ra X$
satisfying
\[
\forall\alpha\in\Lambda\,\exists k_{\alpha}\in[0,1)\,\forall x,y\in X:\ \left|P(x)-P(y)\right|_{\alpha}\le k_{\alpha}|x-y|_{\alpha}.
\]
The notion of contraction has also been extended to uniform spaces
(\cite{Tar74,Tay72}) and to condensing maps on Hausdorff locally
convex linear spaces via the notion of measure of non-compactness
(see e.g.~\cite{BaGo80} and references therein). See also \cite{AgFrOR03}
for a recent survey, and \cite{Fri00,WaZh15,Dud17,WaZhWe19} for updated
references framed in locally convex linear spaces.

In the present section, we want to prove a Banach fixed point theorem
for contractions with loss of derivatives in graded Fréchet spaces.
In this paper, by a \emph{graded Fréchet space} $\left(\mathcal{F},\left(\Vert-\Vert_{k}\right)_{k\in\N}\right)$
we mean a Hausdorff, complete topological vector space whose topology
is defined by an increasing sequence of seminorms: $\Vert-\Vert_{k}\le\Vert-\Vert_{k+1}$
for all $k\in\N$. We denote by $B_{r}^{k}(x):=\left\{ y\in X\mid\Vert x-y\Vert_{k}<r\right\} $
the ball of radius $r\in\R_{>0}$ defined by the $k$-norm.

A first trivial and well known result we will use is the following:
\begin{lem}
\label{lem:Banach_banale}Let $(\mathcal{F},\tau)$ be a topological
space, $P:X\rightarrow\mathcal{F}$ be a continuous function, and
assume that there is $y_{0}\in X$ such that $\exists\,\lim_{n\rightarrow+\infty}P^{n}(y_{0})\in X$.
Then $\lim_{n\rightarrow+\infty}P^{n}(y_{0})$ is a fixed point of
$P.$ In particular, this applies if $\mathcal{F}$ is a Fréchet space
and $\left(P^{n}(y_{0})\right)_{n\in\N}$ is a Cauchy sequence of
points of $X$, where $X\subseteq\mathcal{F}$ is a Cauchy complete
subspace.
\end{lem}

\begin{proof}
The usual proof works: $\overline{y}:=\lim_{n\rightarrow+\infty}P^{n}(y_{0})\in X$
exists by assumption, and since $P:X\ra\mathcal{F}$ is continuous
we have 
\[
P(\overline{y})=P\left(\lim_{n\rightarrow+\infty}P^{n}(y_{0})\right)=\lim_{n\rightarrow+\infty}P^{n+1}(y_{0})=\overline{y}.\qedhere
\]
\end{proof}
In particular, this general Lem.~\ref{lem:Banach_banale} applies
to \emph{contractions with loss of derivatives} in Fréchet spaces:
A key idea in defining this notion is that it has to depend on the
starting point $y_{0}$ of the iterations:
\begin{defn}
\label{def:contractionLoD}Let $\ensuremath{\left(\mathcal{F},\left(\Vert-\Vert_{k}\right)_{k\in\N}\right)}$
be a Fréchet space, $X$ be a closed subset of $\mathcal{F}$, $y_{0}\in X$
and $L\in\N$. We say that $P$ \emph{is a contraction with }$L$
\emph{loss of derivatives starting from }$y_{0}$ (and we simply write
$P\in\LoD\left(X,L,y_{0}\right)$) if the following conditions are
fulfilled:
\begin{enumerate}
\item \label{enu:LoD-P-continuous}$P:X\ra\mathcal{F}$ is continuous;
\item \label{enu:LOD-iterates}$P^{n}(y_{0})\in X$ for all $n\in\N$;
\item \label{enu:LoD-contr}For all $k$, $n\in\N$ there exist $\alpha_{kn}\in\R_{>0}$
such that 
\begin{equation}
\left\Vert P^{n+1}\left(y_{0}\right)-P^{n}\left(y_{0}\right)\right\Vert _{k}\leq\alpha_{kn}\left\Vert P\left(y_{0}\right)-y_{0}\right\Vert _{k+nL},\label{eq:contractionLoD}
\end{equation}
and
\item \label{enu:LoD-W}For all $k\in\N$, the following \emph{Weissinger}
condition holds:
\begin{equation}
\sum_{n=0}^{+\infty}\alpha_{kn}\left\Vert P(y_{0})-y_{0}\right\Vert _{k+nL}<+\infty.\tag{{W}}\label{eq:W}
\end{equation}
\end{enumerate}
\end{defn}

\noindent Note that if we actually have only one norm $\Vert-\Vert_{k}=\Vert-\Vert_{0}$
and $L=0$ (ODE case), then \eqref{eq:W} reduces to the classical
Weissinger condition, \cite{Wei52}.
\begin{flushleft}
We first note that condition \eqref{eq:contractionLoD} trivially
holds for $n=0$ by taking $\alpha_{k,0}=1$. On the other hand, thinking
at \eqref{eq:W}, we are clearly interested only at the asymptotic
behavior of $\alpha_{kn}$ as $n\to+\infty$. Secondly, Def.~\ref{def:contractionLoD}
is weaker than the usual de\-fi\-ni\-tion of contraction because
of the following first three remarks:
\par\end{flushleft}
\begin{enumerate}
\item We have a loss $L\ge0$ of derivatives. If $L=0$, Def.~\ref{def:contractionLoD}
actually tells us that, for all $k\in\N$, there exists $N_{k}\in\N$
such that $\sum_{n=N_{k}}^{+\infty}\alpha_{kn}<1$, namely $P$ is
an ordinary contraction with respect to $\left\Vert -\right\Vert _{k}$
when restricted to $\left\{ P^{n}(x_{0})\mid n\geq N_{k}\right\} $.
\item \label{enu:iterationsX}We will see only in Sec.~\ref{sec:Examples}
that condition Def.~\ref{def:contractionLoD}.\ref{enu:iterationsX}
is not a simple generalization of the usual stronger $P:X\ra X$,
but is essential for the choice of the radii in the PL Thm.~\ref{thm:PLpde}.
\item Both conditions \eqref{eq:contractionLoD} and \eqref{eq:W} depend
on the initial point $y_{0}\in X$. In contrast to the classical BFPT,
this underscores, in an abstract setting, that for PDE the property
to have a contraction with loss of derivatives depends on the initial
condition $y_{0}\in X$. Moreover, in this paper we are solely interested
in existence results for fixed points of contractions with loss of
derivatives; uniqueness results would require conditions closer to
the classical BFPT (see e.g.~Lemma \ref{lem:classicalContr}).
\item Since we want to take $n\to+\infty$, a condition such as Def.~\ref{def:contractionLoD}.\ref{enu:LoD-contr}
intuitively implies that we have to consider all the derivatives controlled
by $\Vert-\Vert_{k}$ for all $k\in\N$. It is for this reason that
in the present work we deal only with smooth solutions of \eqref{eq:PDE}.
\end{enumerate}
More classically, contraction property \eqref{eq:contractionLoD}
is implied by one of the following stronger conditions:
\begin{lem}
\label{lem:classicalContr}Let $X$, $y_{0}$ and $L$ be as in Def.~\ref{def:contractionLoD}.
Then, the following sufficient conditions hold:
\begin{enumerate}
\item \label{enu:classContr_n}If $P:X\ra\mathcal{F}$ satisfies the property
\begin{equation}
\forall k,n\in\N\,\exists\alpha_{kn}\in\R_{>0}\,\forall u,v\in X:\ \left\Vert P^{n}\left(u\right)-P^{n}\left(v\right)\right\Vert _{k}\leq\alpha_{kn}\left\Vert u-v\right\Vert _{k+nL},\label{eq:classContrLoD}
\end{equation}
then condition Def.~\ref{def:contractionLoD}.\ref{enu:LoD-contr}
holds for all $y_{0}\in X$ with the same contraction constants $\alpha_{kn}$.
\item \label{enu:classContr}If \eqref{eq:classContrLoD} holds only for
$n=1$ with $\alpha_{k}:=\alpha_{k1}$, then condition Def.~\ref{def:contractionLoD}.\ref{enu:LoD-contr}
holds for all $y_{0}\in X$ with contraction constants $\tilde{\alpha}_{kn}:=\prod_{j=0}^{n-1}\alpha_{k+jL}$.
\end{enumerate}
Moreover, if $L=0$ and $y_{1}$, $y_{2}$ are fixed points of $P$,
then $\Vert y_{1}-y_{2}\Vert_{k}=0$ for all $k\in\N$. In particular,
if at least one of $||-||_{k}$ is a norm, this entails that $y_{1}=y_{2}$.

\end{lem}

\begin{proof}
\ref{enu:classContr_n}: In fact, \eqref{eq:classContrLoD} yields
\[
\left\Vert P^{n+1}\left(y_{0}\right)-P^{n}\left(y_{0}\right)\right\Vert _{k}=\left\Vert P^{n}\left(P\left(y_{0}\right)\right)-P^{n}\left(y_{0}\right)\right\Vert _{k}\leq\alpha_{kn}\left\Vert P\left(y_{0}\right)-y_{0}\right\Vert _{k+nL}.
\]
Taking $n=1$ in \eqref{eq:classContrLoD}, we have $\Vert P(u)-P(v)\Vert_{k}\le\alpha_{k,1}\Vert u-v\Vert_{k+L}$
and hence $P\left(B_{r/\alpha_{k,1}}^{k}(u)\cap X\right)\subseteq B_{r}^{k}(P(u))\cap X$,
so that $P$ is continuous.

\ref{enu:classContr}: If \eqref{eq:classContrLoD} holds only for
$n=1$, then we can prove the claim by induction on $n$. For $n=0$
the conclusion is trivial since $\tilde{\alpha}_{k0}=1$. For the
inductive step, we have
\begin{align*}
\|P^{n+2}(y_{0})-P^{n+1}(y_{0})\|_{k} & =\|P(P^{n+1}(y_{0}))-P(P^{n}(y_{0}))\|_{k}\\
 & \le\alpha_{k}\|P^{n+1}(y_{0})-P^{n}(y_{0})\|_{k+L}\\
 & \le\alpha_{k}\prod_{j=0}^{n-1}\alpha_{k+L+jL}\|P(y_{0})-y_{0}\|_{k+L+nL}\\
 & =\left(\prod_{j=0}^{n}\alpha_{k+jL}\right)\|P(y_{0})-y_{0}\|_{k+(n+1)L}.
\end{align*}
Continuity of $P$ can be proved as before.

Finally, if $P(y_{l})=y_{l}$, then $\Vert P^{n}(y_{1})-P^{n}(y_{2})\Vert_{k}=\Vert y_{1}-y_{2}\Vert_{k}\le\alpha_{kn}\Vert y_{1}-y_{2}\Vert_{k}\le\Vert y_{1}-y_{2}\Vert_{k}$
because the convergence $\sum_{n=0}^{+\infty}\alpha_{kn}<+\infty$
implies $\alpha_{kn}\le1$ for some $n\in\N$. Thereby, $\Vert y_{1}-y_{2}\Vert_{k}=0$,
which entails $y_{1}=y_{2}$ if $\left\Vert -\right\Vert _{k}$ is
a norm.
\end{proof}
\noindent In Thm.~\ref{thm:Lipschitz for PDE} and in the proof of
Thm.~\ref{thm:PLpde}, we will see that for the normal smooth Cauchy
problem \eqref{eq:PDE}, the corresponding fixed point integral operator
$P$ always satisfies the stronger condition \eqref{eq:classContrLoD}.
Therefore, in all these cases the real dependence on $y_{0}$ actually
lies in conditions \eqref{eq:W} and \ref{enu:iterationsX}.

Even in the simple case of the transport equation $\partial_{t}y=c\cdot\partial_{x}y$,
where $c$, $y\in\mathcal{C}^{\infty}([0,a]\times S)$, $S\Subset\R$,
we can recognize the appearance of a loss of derivative $L=1$ due
to the occurrence of the term $\partial_{x}y$ on the right hand side
of the PDE. In fact, set $P(u)(t,x):=y_{0}(x)+\int_{0}^{t}c(s,x)\cdot\partial_{x}u(s,x)\,\diff{s}$
for a fixed $y_{0}\in\mathcal{C}^{\infty}(S)$ and for any $u\in\mathcal{C}^{\infty}([0,a]\times S)$.
Considering on $\mathcal{C}^{\infty}([0,a]\times S)$ the family of
norms
\[
\Vert u\Vert_{k}:=\max_{\substack{\left|\alpha\right|+\left|\beta\right|\le k}
}\max_{(t,x)\in[0,a]\times S}\left|\partial_{t}^{\alpha}\partial_{x}^{\beta}u(t,x)\right|,
\]
we would like to argue in the following way (where, for simplicity,
we consider only the case $n=1$ in property \eqref{eq:classContrLoD}):
\begin{align}
\Vert P(u)-P(v)\Vert_{k} & =\left\Vert \int_{0}^{t}c\cdot\partial_{x}(u-v)\,\diff{s}\right\Vert _{k}\le\nonumber \\
 & \le a\cdot\Vert c\Vert_{k}\cdot\Vert\partial_{x}(u-v)\Vert_{k}\le a\cdot\Vert c\Vert_{k}\cdot\Vert u-v\Vert_{k+1}.\label{eq:exWrong}
\end{align}
The problem in this deduction is that the first inequality is generally
\emph{false} for these norms: if $k>0$ then any derivative $\partial_{t}$
deletes the integration, so that the factor $a$ (which is important
to get a local solution) cannot appear in \eqref{eq:exWrong} (see
Sec.~\ref{subsec:SupNormsInt} for more details). We will fix this
problem by taking another family of norms which, anyway, respect the
same basic ideas (see Def.~\ref{def:FrechetFncs}), and where the
estimates \eqref{eq:exWrong} hold.

Def.~\ref{def:contractionLoD} has been tuned to allow the proof
of the following result, whose proof is surprisingly simple:
\begin{thm}[BFPT with loss of derivatives]
\label{thm:BFPTLoss}In the assumptions of Def.~\ref{def:contractionLoD},
if $P\in\LoD\left(X,L,y_{0}\right)$, then $\left(P^{n}(y_{0})\right)_{n\in\N}$
is a Cauchy sequence, and hence
\[
\overline{y}:=\lim_{n\rightarrow+\infty}P^{n}(y_{0})\in X
\]
is a fixed point of $P$. Moreover, for all $k$, $n\in\N$ we have
that 
\[
\left\Vert \overline{y}-P^{n}\left(y_{0}\right)\right\Vert _{k}\leq\sum_{j=n}^{+\infty}\alpha_{kj}\left\Vert P\left(y_{0}\right)-y_{0}\right\Vert _{k+jL}.
\]
\end{thm}

\begin{proof}
If we prove that $\left(P^{n}(y_{0})\right)_{n\in\N}$ is a Cauchy
sequence, the claim follows from Lem.~\ref{lem:Banach_banale}. Let
$m$, $n$, $k\in\N$ with $m>n$. Then
\begin{align}
\left\Vert P^{m}\left(y_{0}\right)-P^{n}\left(y_{0}\right)\right. & \left.\!\!\right\Vert _{k}\leq\left\Vert P^{m}\left(y_{0}\right)-P^{m-1}\left(y_{0}\right)\right\Vert _{k}+\dots+\left\Vert P^{n+1}\left(y_{0}\right)-P^{n}\left(y_{0}\right)\right\Vert _{k}\nonumber \\
\phantom{\!\!\!\!}\le & \alpha_{k,m-1}\left\Vert P\left(y_{0}\right)-y_{0}\right\Vert _{k+(m-1)L}+\dots+\alpha_{kn}\left\Vert P\left(y_{0}\right)-y_{0}\right\Vert _{k+nL}\nonumber \\
\phantom{\!\!\!\!}= & \sum_{j=n}^{m-1}\alpha_{kj}\left\Vert P\left(y_{0}\right)-y_{0}\right\Vert _{k+jL}.\label{eq:cauchyP}
\end{align}

We conclude using \eqref{eq:W} of Def.~\ref{def:contractionLoD}.
The final claim holds by taking $m\to+\infty$ in \eqref{eq:cauchyP}
as $\overline{y}=\lim_{m\rightarrow+\infty}P^{m}(y_{0})$.
\end{proof}
\noindent Clearly, the chain of inequalities in \eqref{eq:cauchyP}
could be stopped in several different ways. For example, as $\left\Vert P\left(y_{0}\right)-y_{0}\right\Vert _{k+jL}\le\left\Vert P\left(y_{0}\right)-y_{0}\right\Vert _{k+(m-1)L}$,
we can continue arriving at a final term of the form $\left\Vert P\left(y_{0}\right)-y_{0}\right\Vert _{k+(m-1)L}\cdot\sum_{j=n}^{+\infty}\alpha_{kj}$.
Actually, this would lead us to consider a limit of the form $\lim_{\substack{n,m\to+\infty\\
n\le m
}
}p_{m}\cdot q_{n}$, which never exists if $p_{m}\to+\infty$ and $q_{n}=a^{n}$, because
we can take $n\to+\infty$ depending on $p_{m}$. On the contrary,
in condition \eqref{eq:W} the summation index $n$ links the two
factors in the series; looking at Lem.~\ref{lem:classicalContr}
and next Thm.~\ref{thm:Lipschitz for PDE}, we can also say that
condition \eqref{eq:W} links the right hand side $F$ and the initial
conditions $y_{0j}$ of the Cauchy problem \eqref{eq:PDE}. On the
other hand, it is clear that the proof of previous Thm.~\ref{thm:BFPTLoss}
is quite standard, and this underscores that the key step lies in
Def.~\ref{def:contractionLoD} of contraction with loss of derivatives
$L$ starting from $y_{0}$.

\section{\label{sec:equationsLoD}Solutions of equations}

In this Section, we want to use the Banach fixed point Thm.~\ref{thm:BFPTLoss}
with loss of derivatives to solve equations of the form $f(x)=y$
in arbitrary graded Fréchet spaces. We can also inscribe this problem
as the proof of local surjection in inverse function theorems.

In the following, given a sequence $R=\left(r_{k}\right)_{k\in\N}$,
$r_{k}\in\mathbb{R}_{>0}\cup\{+\infty\}$, and a point $x_{0}$ in
a graded Fréchet space $\mathcal{F}$, we set 
\begin{equation}
\bar{B}_{R}\left(x_{0}\right):=\left\{ x\in\mathcal{F\mid\,}\forall k\in\N:\ \left\Vert x_{0}-x\right\Vert _{k}\leq r_{k}\right\} .\label{eq:ball-y_0-R}
\end{equation}

\noindent Note that $\bar{B}_{R}(x_{0})$ is closed in $\mathcal{F}$
as it is a countable intersection of closed sets. Moreover, $\bar{B}_{R}(x_{0})$
trivially generalizes the space usually used in the proof of the PLT
for ODE, where we only have $r_{k}=r_{0}<+\infty$. The first trivial
consequence of Thm.~\ref{thm:BFPTLoss} reformulates $f(x)=y$ as
a fixed point of the map $P(x):=x-f(x)+y$:
\begin{cor}
\label{cor:solEqLoD}Let $\left(\mathcal{F},\left(\Vert-\Vert_{k}\right)_{k\in\N}\right)$
be a graded Fréchet space. Let $X$ be a closed subset of $\mathcal{F}$,
$f:X\ra\mathcal{F}$ be a continuous map, $y_{0}\in\mathcal{F}$ and
$L\in\N$. Set $P(x):=x-f(x)+y_{0}$ and assume that for all $k$,
$n\in\N$, we have
\begin{equation}
P^{n}(y_{0})\in X,\label{eq:iterationsSolEq}
\end{equation}
\begin{align}
\exists\alpha_{kn}\in\R_{>0}:\ \left\Vert P^{n+1}\left(y_{0}\right)-P^{n}\left(y_{0}\right)\right\Vert _{k} & \leq\alpha_{kn}\left\Vert P\left(y_{0}\right)-y_{0}\right\Vert _{k+nL},\label{eq:losCor}\\
\sum_{n=0}^{+\infty}\alpha_{kn}\cdot\left\Vert P\left(y_{0}\right)-y_{0}\right\Vert _{k+nL} & <+\infty,\label{eq:WCor}
\end{align}
then, there exists $x_{0}\in X$ such that $f(x_{0})=y_{0}$.
\end{cor}

\noindent In spite of its triviality, we will see in Sec.~\ref{sec:Examples}
that this result allows us to solve PDE with the same scope of the
next PL Thm.~\ref{thm:PLpde} (which, on the other hand, already
includes in its proof the verification of property \eqref{eq:losCor}).
Moreover, let us now note that in Cor.~\ref{cor:solEqLoD} we do
not require differentiability of $f$ let alone the existence of some
inverse of its differential $\diff f(x)$.

On the other hand, simply by generalizing the derivation of the inverse
function theorem from the classical BFPT in Banach spaces (see e.g.~\cite{How97,Cla,Lea}),
we obtain the following theorem:
\begin{thm}
\label{thm:IFT}Let $\left(\mathcal{X},\left(\Vert-\Vert_{k}\right)_{k\in\N}\right)$,
$\left(\mathcal{Y},\left(|-|_{k}\right)_{k\in\N}\right)$ be graded
Fréchet spaces, $x_{0}\in\mathcal{X}$, $R=\left(r_{k}\right)_{k\in\N}$,
$r_{k}\in\mathbb{R}_{>0}\cup\{+\infty\}$ and set $\bar{r}_{k+L_{D}}:=\frac{r_{k+L}(1-\alpha_{k})}{\delta_{k+L_{D}}}$,
$\bar{R}:=(\bar{r}_{k+L_{D}})_{k\in\N}$. Let $D:\mathcal{Y}\ra\mathcal{X}$
\emph{(dextrum)} and $S:\mathcal{X}\ra\mathcal{Y}$ \emph{(sinistrum)}
be linear maps, $f:\bar{B}_{R}(x_{0})\ra\mathcal{Y}$ and set $P_{y}(x):=x-D\left[f(x)-y\right]$
for all $y\in\bar{B}_{\bar{R}}(f(x_{0}))$ and all $x\in\bar{B}_{R}(x_{0})$.
Assume that for all $k\in\N$ we have:
\begin{enumerate}
\item \label{enu:contrLoD}$\|D(f(x))-D(f(\bar{x}))-(x-\bar{x})\|_{k}\le\alpha_{k}\cdot\|x-\bar{x}\|_{k+L}$
for some $L\in\N$, $\alpha_{k}>0$ and all $x$, $\bar{x}\in\bar{B}_{R}(x_{0})$;
\item \label{enu:rightInverse}$D$ is the right inverse of $S$,i.e.~$S\circ D=1_{\mathcal{Y}}$;
\item \label{enu:Scont}$|S(x)-S(\bar{x})|_{k}\le\sigma_{k}\cdot\|x-\bar{x}\|_{k+L_{S}}$
for some $L_{S}\in\N$, some $\sigma_{k}>0$ and all $x$, $\bar{x}\in\mathcal{X}$;
\item \label{enu:Dcont}$\|D(y)-D(\bar{y})\|_{k}\le\delta_{k+L_{D}}\cdot|y-\bar{y}|_{k+L_{D}}$
for some $L_{D}\in\N$, some $\delta_{k}>0$ and all $y$, $\bar{y}\in\mathcal{Y}$;
\item \label{enu:contrLoDmin}$\|P_{y}^{n+1}(x_{0})-P_{y}^{n}(x_{0})\|_{k}\le\alpha_{kn}\|P_{y}(x_{0})-x_{0}\|_{k+nL}$
for all $n\in\N$, and some $\alpha_{kn}>0$;
\item \label{enu:iterationsIFT}$P_{y}^{n}(x_{0})\in\bar{B}_{R}(x_{0})$
for all $n\in\N$ and all $y\in\bar{B}_{\bar{R}}(f(x_{0}))$.
\item \label{enu:WIFT}$\sum_{n=0}^{+\infty}\alpha_{kn}\cdot\|P_{y}(x_{0})-x_{0}\|_{k+nL}<+\infty$.
\end{enumerate}
Then, the following properties hold:
\begin{enumerate}[resume]
\item \label{enu:f_Lip}$|f(x)-f(\bar{x})|_{k}\le\sigma_{k}\left(\alpha_{k}\|x-\bar{x}\|_{k+L_{S}+L}+\|x-\bar{x}\|_{k+L_{S}}\right)$
for all $k\in\N$, $x$, $\bar{x}\in\bar{B}_{R}(x_{0})$, i.e.~$f$
is locally Lipschitz with loss of derivatives;
\item \label{enu:f_lowBound}$|f(x)-f(\bar{x})|_{k+L_{D}}\ge\frac{1}{\delta_{k+L_{D}}}\left(\|x-\bar{x}\|_{k}-\alpha_{k}\|x-\bar{x}\|_{k+L}\right)$
for all $k\in\N$, $x$, $\bar{x}\in\bar{B}_{R}(x_{0})$;
\item \label{enu:soleqLoD}$\forall y\in\bar{B}_{\bar{R}}\left(f\left(x_{0}\right)\right)\exists x\in\bar{B}_{R}\left(x_{0}\right):\ f(x)=y$.
\item \label{enu:contInv}If $g:\bar{B}_{\bar{R}}\left(f\left(x_{0}\right)\right)\ra\bar{B}_{R}\left(x_{0}\right)$
is any right-inverse of $f$, i.e.~$f\circ g=1_{\bar{B}_{\bar{R}}(f(x_{0}))}$,
and $\left(\mathcal{X},\left(\Vert-\Vert_{k}\right)_{k\in\N}\right)=\left(\mathcal{X},\|-\|\right)$
is a Banach space, then $\|g(y)-g(\bar{y})\|\le\frac{\delta_{k+L_{D}}}{1-\alpha_{k}}\cdot|y-\bar{y}|_{k+L_{D}}$
for all $k\in\N$, $y$, $\bar{y}\in\bar{B}_{\bar{R}}(f(x_{0}))$.
\end{enumerate}
In particular: 1) if $r_{k+L}\le r_{k}$ for all $k\in\N$ then $P_{y}:\bar{B}_{R}(x_{0})\ra\bar{B}_{R}(x_{0})$
for all $y\in\bar{B}_{\bar{R}}(f(x_{0}))$, and hence assumption \ref{enu:iterationsIFT}
always holds; 2) if we set $\alpha_{kn}:=\prod_{j=0}^{n-1}\alpha_{k+jL}$,
then \ref{enu:contrLoDmin} always holds.
\end{thm}

\noindent Note that we do not require that the spaces $\mathcal{X}$,
$\mathcal{Y}$ are tame or admit smoothing operators like in the Nash-Moser
theorem, see \cite{Ham82}; we also do not require that $f$ is differentiable
as in Ekeland inverse function theorem \cite{Eke11}. Moreover, its
proof is a generalization of \cite{How97}, so that it includes the
inverse function theorem for Lipschitz maps in Banach spaces if $L=L_{D}=L_{S}=0$
and $r_{k}=r_{0}$ for all $k\in\N$.
\begin{proof}
\ref{enu:f_Lip}: From \ref{enu:rightInverse} and \ref{enu:Scont},
we can write
\begin{align*}
|f(x)-f(\bar{x})|_{k} & =|S(D(f(x))-S(D(f(\bar{x}))|_{k}\le\sigma_{k}\|D(f(x))-D(f(\bar{x}))\|_{k+L_{S}}\\
 & =\sigma_{k}\|D(f(x))-D(f(\bar{x}))-(x-\bar{x})+(x-\bar{x})\|_{k+L_{S}}\\
 & \le\sigma_{k}\left(\alpha_{k}\|x-\bar{x}\|_{k+L_{S}+L}+\|x-\bar{x}\|_{k+L_{S}}\right),
\end{align*}
where we used \ref{enu:contrLoDmin}.

\noindent \ref{enu:f_lowBound}: Once again from \ref{enu:contrLoDmin},
we get
\begin{align*}
\|x-\bar{x}\|_{k} & \le\|D(f(x))-D(f(\bar{x}))-(x-\bar{x})\|_{k}+\|D(f(x))-D(f(\bar{x}))\|_{k}\\
 & \le\alpha_{k}\|x-\bar{x}\|_{k+L}+\delta_{k+L_{D}}|f(x)-f(\bar{x})|_{k+L_{D}}
\end{align*}
because of \ref{enu:Dcont}.

We immediately note that \ref{enu:f_Lip} and \ref{enu:Dcont} imply
the continuity of $f$ and $D$ resp. and hence also of the map $P_{y}(x):=x-D\left[f(x)-y\right]$,
$P_{y}:\bar{B}_{R}(x_{0})\ra\mathcal{X}$. We have $P_{y}(x)=x$ if
and only if $D\left[f(x)-y\right]=0$ and hence if and only if $f(x)=y$
from \ref{enu:rightInverse}. Assumption \ref{enu:contrLoDmin} corresponds
exactly to the request that this map contracts with loss of derivatives
$L$. Weissinger condition is assumption \ref{enu:WIFT}. This proves
that $P_{y}$ is a contraction with loss of derivatives $L$, hence
\ref{enu:soleqLoD} holds.

In particular, if $r_{k+L}\le r_{k}$ we can that prove that $P_{y}:\bar{B}_{R}(x_{0})\ra\bar{B}_{R}(x_{0})$.
In fact, if $y\in\bar{B}_{S}(f(x_{0}))$ and $x\in\bar{B}_{R}(x_{0})$,
then
\begin{align}
\|P_{y}(x)-x_{0}\|_{k} & =\|x-x_{0}-D\left[f(x)-f(x_{0})\right]-D\left[f(x_{0})-y\right]\|_{k}\nonumber \\
 & \le\|x-x_{0}-D\left[f(x)-f(x_{0})\right]\|_{k}+\|D(f(x_{0}))-D(y)\|_{k}\nonumber \\
 & \le\alpha_{k}\|x-x_{0}\|_{k+L}+\delta_{k+L_{D}}|f(x_{0})-y|_{k+L_{D}}\le\nonumber \\
 & \le\alpha_{k}\cdot r_{k+L}+\delta_{k+L_{D}}\cdot\frac{r_{k+L}(1-\alpha_{k})}{\delta_{k+L_{D}}}=r_{k+L}\le r_{k},\label{eq:contrBall}
\end{align}
where we used $|f(x_{0})-y|_{k+L_{D}}\le\bar{r}_{k+L_{D}}=\frac{r_{k+L}(1-\alpha_{k})}{\delta_{k+L_{D}}}$.
Finally, \ref{enu:contrLoD} yields:
\begin{align*}
\|P_{y}(x)-P_{y}(\bar{x})\|_{k} & =\|x-D\left[f(x)-y\right]-\bar{x}+D\left[f(\bar{x})-y\right]\|_{k}\\
 & =\|D\left[f(x)-f(\bar{x})\right]-(x-\bar{x})\|_{k}\\
 & \le\alpha_{k}\|x-\bar{x}\|_{k+L}.
\end{align*}
From Lem.~\ref{lem:classicalContr}.\ref{enu:classContr} and assumption
\ref{enu:iterationsIFT}, we hence obtain that $P_{y}$ satisfies
Def.~\ref{def:contractionLoD}.\ref{enu:LoD-contr} with contraction
constants $\alpha_{kn}:=\prod_{j=0}^{n-1}\alpha_{k+jL}$.

\ref{enu:contInv}: This follows directly from \ref{enu:f_lowBound}
for $x:=g(y)$ and $\bar{x}:=g(\bar{y})$ considering that $\|-\|_{k}=\|-\|_{k+L}=\|-\|$.
\end{proof}
\noindent Even if the previous statement allows us to take $r_{k}=+\infty$,
it is now Weissinger condition \ref{enu:WIFT} that forces to take
$y$ near $f(x_{0})$: the factor $\|P_{y}(x_{0})-x_{0}\|_{k+nL}=\|D\left[f(x_{0})-y\right]\|_{k+nL}$
is small if $y$ is near $f(x_{0})$; note also that \ref{enu:WIFT}
is implied by the stronger condition $\sum_{n=0}^{+\infty}\alpha_{kn}r_{k+nL}<+\infty$
because of \ref{enu:iterationsIFT}. On the other hand, assumption
\ref{enu:WIFT} is in principle compatible with growing term $\|P_{y}(x_{0})-x_{0}\|_{k+nL}$
as $k+nL\to+\infty$, even if the Lipschitz factors $\alpha_{kn}$
must keep the series convergent.

\section{\label{sec:PLTpde}A Picard-Lindelöf theorem for PDE}

In the following, considering the Cauchy problem \eqref{eq:PDE},
we always set and assume
\begin{align*}
\hat{L} & :=\text{Card\ensuremath{\left\{ (\alpha,\gamma)\in\N^{s}\times\N_{\le p}\mid|\alpha|\leq L\right\} } }\\
a & ,\ b\in\R_{>0},\quad[t_{0}-a,t_{0}+b]\times S=:T\times S\Subset\R^{1+s}\\
F & \in\mC^{\infty}\left(T\times S\times\R^{m\cdot\hat{L}},\R^{m}\right)\\
y_{0j} & \in\mC^{\infty}\left(S,\R^{m}\right)\quad\forall j=0,\ldots,d-1,
\end{align*}
where $p\le d-1$ denotes the maximum order of derivatives $\partial_{t}^{\gamma}y(t,x)\in\R^{m}$
appearing on the right hand side of \eqref{eq:PDE}.

\subsection{\label{subsec:SupNormsInt}Supremum norms of integral functions}

We have already mentioned that the first inequality in \eqref{eq:exWrong}
is generally wrong. Let us construct a counter example for the space
of one variable functions $\Coo([0,a])$ with the norms $\Vert u\Vert_{k}:=\max_{\substack{h\le k\\
t\in[0,a]
}
}\left|u^{(h)}(t)\right|$. Take e.g.~$a=\frac{1}{2}$ and consider the straight line $y=1$.
Then
\begin{align*}
\left\Vert \int_{0}^{(-)}y\right\Vert _{1} & =\max\left(\max_{t\in[0,\frac{1}{2}]}\left|\int_{0}^{t}1\right|,\max_{t\in[0,\frac{1}{2}]}\left|1\right|\right)=1
\end{align*}
and
\[
\Vert y\Vert_{1}=\max\left(\max_{t\in[0,\frac{1}{2}]}\left|1\right|,\max_{t\in[0,\frac{1}{2}]}\left|0\right|\right)=1,
\]
 therefore
\[
\left\Vert \int_{0}^{(-)}1\right\Vert _{1}=1>a\cdot\Vert y\Vert_{1}=\frac{1}{2}.
\]
 It is not hard to prove that, actually, $\Vert\int_{0}^{(-)}y\Vert_{k}>a\cdot\Vert y\Vert_{k}$
for all $k\geq1$.

This remark allows us to understand, once again, why in the classical
proof of the smooth PLT we consider only the space $\mathcal{C}^{0}([0,a])$
of continuous functions with the supremum norm $\Vert-\Vert_{0}$:
in fact, even if we aim to get a \emph{smooth solution} $y$ (so that
we would have to control all its derivatives), the normal form of
the equation recursively yields the smoothness of $y$ starting from
a continuous solution of the corresponding integral problem.

Similarly, we can argue for normal PDE: considering the corresponding
integral problem 
\begin{align}
y(t,x) & =i_{0}(t,x)+\int_{t_{0}}^{t}\diff s_{d}\ptind^{d}\int_{t_{0}}^{s_{2}}F\left[s_{1},x,\left(\partial_{x}^{\alpha}\partial_{t}^{\gamma}y\right)_{\substack{|\alpha|\leq L\\
\gamma\le p
}
}\right]\diff s_{1},\label{eq:PDEint}\\
i_{0}(t,x): & =\sum_{j=0}^{d-1}\frac{y_{0j}(x)}{j!}(t-t_{0})^{j}.\label{eq:i_0}
\end{align}
we only need that the function $y$ is of class $\mathcal{C}^{p}$
in $t$ and smooth in $x$: smoothness in $t$ recursively follows
from \eqref{eq:PDEint}, and we only have to control all its derivatives
in $x$. This motivates the introduction of a space with this kind
of functions.

\subsection{Spaces of separately regular functions}

As we mentioned above, instead of considering functions which are
jointly regular in both variables $(t,x)$, we need to consider separate
degree of regularity in each variable.
\begin{defn}
\label{def:FrechetFncs}\ 
\begin{enumerate}
\item If $X\subseteq\R^{n}$ is an arbitrary subset and $q\in\N\cup\{\infty\}$,
we say that $f\in\mathcal{C}^{q}(X,\R^{m})$ if for each $x\in X$
there exists an open neighborhood $x\in U\subseteq\R^{n}$ and a function
$F\in\mathcal{C}^{q}(U,\R^{m})$ such that $F|_{U\cap X}=f|_{U\cap X}$.
\item Let $T\times S\Subset\mathbb{R}^{1+s}$. Set
\begin{equation}
\N_{p}^{1+s}:=\left\{ \beta\in\N^{1+s}\mid\beta_{1}\le p\right\} ,\label{eq:N_p1+s}
\end{equation}
and denote by $\mC_{t}^{p}\mC_{x}^{\infty}\left(T\times S,\R^{m}\right)$
the set of continuous functions $y\in\mathcal{C}^{0}(T\times S,\R^{m})$
such that
\[
\forall\beta\in\N_{p}^{1+s}:\ \exists\,\partial^{\beta}y\in\mathcal{C}^{0}(T\times S,\R^{m}).
\]
The functions in $\mC_{t}^{p}\mC_{x}^{\infty}\left(T\times S,\R^{m}\right)$
are called \emph{separately }$\mC_{t}^{p}\mC_{x}^{\infty}$\emph{
regular.} This space is endowed with the countable family of norms
$\left\Vert -\right\Vert _{k}$, $k\in\N$, defined by
\begin{equation}
\Vert y\Vert_{k}:=\max_{1\leq h\leq m}\underset{\substack{|\beta|\le k\\
\beta\in\N_{p}^{1+s}
}
}{\max}\max_{(t,x)\in T\times S}\left|\partial^{\beta}y^{h}(t,x)\right|.\label{eq:norms}
\end{equation}
\end{enumerate}
\end{defn}

In problem \eqref{eq:PDE}, we could also consider a reduction to
first order: setting $y^{1}:=y$, $y^{j+1}:=\partial_{t}y^{j}$, $j=1,\ldots,p$,
problem \eqref{eq:PDE} is equivalent to
\begin{equation}
\begin{cases}
\partial_{t}Y(t,x)=\bar{F}\left[t,x,\left(\partial_{x}^{\alpha}Y^{\gamma+1}\right)_{\substack{|\alpha|\leq L\\
\gamma\le p
}
}\right],\\
Y(t_{0},x)=Y_{0}(x),
\end{cases}\label{eq:PDEred}
\end{equation}
where, as usual, we mean $\partial_{x}^{\alpha}Y^{\gamma+1}=\partial_{x}^{\alpha}Y^{\gamma+1}(t,x)$,
and
\begin{align*}
 & Y(t,x):=(y^{1}(t,x),\ldots,y^{p+1}(t,x))\\
 & Y_{0}(x):=(y_{0}^{0}(x),\ldots,y_{0}^{d-1}(x))\\
 & \bar{F}^{d}\left[t,x,\left(u^{\alpha,\gamma}\right)_{\substack{|\alpha|\leq L\\
\gamma\le p
}
}\right]:=F\left[t,x,\left(u^{\alpha,\gamma}\right)_{\substack{|\alpha|\leq L\\
\gamma\le p
}
}\right]\\
 & \bar{F}^{j}\left[t,x,\left(u^{\alpha,\gamma}\right)_{\substack{|\alpha|\leq L\\
\gamma\le p
}
}\right]:=y^{j+1}
\end{align*}
for $j=1,\ldots,p$. In the corresponding integral problem \eqref{eq:PDEint},
we could assume $d=1$ and hence we only need that the function $y$
is of class $\mC_{t}^{0}\mC_{x}^{\infty}$. On the one hand, this
would simplify our next statements. However, we would obtain a PLT
with assumptions that are clear only for $d=1$, and to prove from
this a corresponding result for $d>1$ is not so easy. For this reason,
we prefer to directly proceed with the generic problem \eqref{eq:PDEint}
without implementing a reduction to first order.
\begin{lem}
In the notations of Def.~\ref{def:FrechetFncs}, $\left(\mC_{t}^{p}\mC_{x}^{\infty}\left(T\times S,\R^{m}\right),\left(\left\Vert -\right\Vert _{k}\right)_{k\in\N}\right)$
is a graded Fréchet space.
\end{lem}

\begin{proof}
The only non trivial property to check is that the topology induced
by the family of norms $\left(\left\Vert -\right\Vert _{k}\right)_{k\in\N}$
is Cauchy complete. Let $\left(y_{n}\right)_{n\in\N}$ be a Cauchy
sequence in $\left(\mC_{t}^{p}\mC_{x}^{\infty}\left(T\times S,\R^{m}\right),\left(\left\Vert -\right\Vert _{k}\right)_{k\in\N}\right)$,
so that for $k=0$, the sequence $\left(y_{n}\right)_{n\in\N}$ converges
uniformly. Let $y:T\times S\rightarrow\R^{m}$ be the continuous function
defined by
\begin{equation}
y(t,x):=\lim_{n\rightarrow+\infty}y_{n}(t,x)\quad\forall(t,x)\in T\times S.\label{eq:limit_f}
\end{equation}
For all $\beta\in\N_{p}^{1+s}$, we have $\Vert\partial^{\beta}y_{l}-\partial^{\beta}y_{n}\Vert_{0}\le\Vert y_{l}-y_{n}\Vert_{h}$
and hence $\left(\partial^{\beta}y_{n}\right)_{n\in\N}$ is a uniformly
convergent Cauchy sequence in $\mC^{0}\left(T\times S,\R^{m}\right)$
that converges to $\partial^{\beta}y\in\mC^{0}\left(T\times S,\R^{m}\right)$.
This shows that $y\in\mC_{t}^{p}\mC_{x}^{\infty}(T\times S,\R^{m})$.
It remains to prove that $y_{n}\to y$ with respect to the norms defined
in \eqref{eq:norms}. For $k=0$, we simply recall that the limit
in \eqref{eq:limit_f} is actually a uniform limit. For $k>0$, we
note that for all $\beta\in\N_{p}^{1+s}$ with $|\beta|\le k$, the
sequence $\left(\partial^{\beta}y_{n}\right)_{n\in\N}$ converges
uniformly in $T\times S$ to $\partial^{\beta}y$.
\end{proof}
\noindent Exactly because we have $\beta_{1}\le p<d$ in \eqref{eq:N_p1+s},
we can now have the desired estimate in considering the norm of an
integral function:
\begin{lem}
\label{lem:normsInt}Let $f\in\mC_{t}^{0}\mC_{x}^{\infty}(T\times S,\R^{m})$
and, for every $k\in\N$, let $M_{k}\in\mC^{0}(T\times S)$ be such
that
\begin{equation}
\left|\partial_{x}^{\nu}f^{h}(t,x)\right|\le M_{k}(t,x)\label{eq:M_i functions}
\end{equation}
for all \textup{$(t,x)\in T\times S$, $h=1,\ldots,m$, and all $\nu\in\N^{s}$
such that }$|\nu|\le k$. Set
\[
\bar{M}_{kj}(t,x):=\left|\int_{t_{0}}^{t}\diff s_{d}\ptind^{j}\int_{t_{0}}^{s_{2}}M_{k}(s_{1},x)\,\diff{s_{1}}\right|\quad\forall(t,x)\in T\times S\,\forall j=1,\ldots,d.
\]
Then, with respect to the norms in the space $\mC_{t}^{p}\mC_{x}^{\infty}$
defined as in \eqref{eq:norms}, we have
\begin{enumerate}
\item \label{enu:normsIntMax}$\left\Vert {\displaystyle \int_{t_{0}}^{(-)}}\diff s_{d}\ptind^{d}\int_{t_{0}}^{s_{2}}f(s_{1},-)\,\diff{s_{1}}\right\Vert _{k}\le{\displaystyle \max_{\substack{x\in S\\
0<j\le d
}
}}\bar{M}_{kj}(t_{0}+\max(a,b),x)$.
\end{enumerate}
In particular, if $M_{k}=\|f\|_{k}$:
\begin{enumerate}[resume]
\item \label{enu:normsIntAlpha}$\left\Vert {\displaystyle \int_{t_{0}}^{(-)}}\diff s_{d}\ptind^{d}\int_{t_{0}}^{s_{2}}f(s_{1},-)\,\diff{s_{1}}\right\Vert _{k}\le\max(a,b)\cdot\Vert f\Vert_{k}$.
\end{enumerate}
\end{lem}

\begin{proof}
\ref{enu:normsIntMax}: Clearly, the notation $\int_{t_{0}}^{(-)}\diff s_{d}\ptind^{d}{\textstyle \int_{t_{0}}^{s_{2}}}f(s_{1},-)\,\diff{s_{1}}$
denotes the function
\[
\left((t,x)\in T\times S\mapsto{\displaystyle \int_{t_{0}}^{t}}\diff s_{d}\ptind^{d}\int_{t_{0}}^{s_{2}}f(s_{1},x)\,\diff{s_{1}}\in\R^{m}\right)\in\mC_{t}^{p}\mC_{x}^{\infty}(T\times S,\R^{m})
\]
(actually, this is a $\mC_{t}^{d}\mC_{x}^{\infty}$-function, but
in the statement we are considering the norms $\|-\|_{k}$ of the
space $\mC_{t}^{p}\mC_{x}^{\infty}$). For some $\beta\in\N_{p}^{1+s}$
with $|\beta|\le k$, and some $h=1,\ldots,m$, we have
\begin{multline}
\left\Vert {\displaystyle \int_{t_{0}}^{(-)}}\diff s_{d}\ptind^{d}\int_{t_{0}}^{s_{2}}f(s_{1},-)\,\diff{s_{1}}\right\Vert _{k}=\\
=\max_{(t,x)\in T\times S}\left|\partial^{\beta}{\displaystyle \int_{t_{0}}^{t}}\diff s_{d}\ptind^{d}\int_{t_{0}}^{s_{2}}f^{h}(s_{1},x)\,\diff{s_{1}}\right|.\label{eq:maxDer}
\end{multline}
But $\beta_{1}\le p<d$, and hence, setting $\nu:=(\beta_{2},\dots,\beta_{s})$,
the operator $\partial^{\beta}=\partial_{x}^{\nu}\partial_{t}^{\beta_{1}}$
deletes $\beta_{1}$ integrals in \eqref{eq:maxDer}; set $\bar{\jmath}:=d-\beta_{1}>0$.
Differentiation under the integral sign yields
\begin{align*}
\left\Vert {\displaystyle \int_{t_{0}}^{(-)}}\diff s_{d}\ptind^{d}\right. & \left.\int_{t_{0}}^{s_{2}}f(s_{1},-)\,\diff{s_{1}}\right\Vert _{k}=\\
 & =\max_{(t,x)\in T\times S}\left|{\displaystyle \int_{t_{0}}^{t}}\diff s_{\bar{\jmath}}\ptind^{\bar{\jmath}}\int_{t_{0}}^{s_{2}}\partial_{x}^{\nu}f^{h}(s_{1},x)\,\diff{s_{1}}\right|\\
 & \le\max_{(t,x)\in T\times S}\text{sgn}(t-t_{0})^{\bar{\jmath}}{\displaystyle \int_{t_{0}}^{t}}\diff s_{\bar{\jmath}}\ptind^{\bar{\jmath}}\int_{t_{0}}^{s_{2}}\left|\partial_{x}^{\nu}f^{h}(s_{1},x)\right|\,\diff{s_{1}}\\
 & \le\max_{(t,x)\in T\times S}\text{sgn}(t-t_{0})^{\bar{\jmath}}{\displaystyle \int_{t_{0}}^{t}}\diff s_{\bar{\jmath}}\ptind^{\bar{\jmath}}\int_{t_{0}}^{s_{2}}M_{k}(s_{1},x)\,\diff{s_{1}}\\
 & =\max_{(t,x)\in T\times S}\bar{M}_{k\bar{\jmath}}(t,x)\\
 & \le\max_{\substack{x\in S\\
0<j\le d
}
}\bar{M}_{kj}(t_{0}+\max(a,b),x).
\end{align*}

\noindent Note that if $t>t_{0}$, then $t_{0}<s_{\bar{\jmath}}<t$;
if $t<t_{0}$, then $t<s_{\bar{\jmath}}<t_{0}$, and in both cases
$\text{sgn}(t-t_{0})=\text{sgn}(s_{\bar{\jmath}}-t_{0})$. Similarly,
we can proceed for the other integration variables $s_{q}$.

\ref{enu:normsIntAlpha}: Condition \eqref{eq:M_i functions} holds
if $M_{k}=\|f\|_{k}$, and we have
\[
\bar{M}_{kj}(t,x)=\|f\|_{k}\frac{(t-t_{0})^{j}}{j!}.
\]
Thereby, $\max_{\substack{x\in S\\
j\le d
}
}\bar{M}_{kj}(t_{0}+\max(a,b),x)=\max(a,b)\|f\|_{k}$.
\end{proof}
To solve problem \eqref{eq:PDE} or, equivalently, the integral problem
\eqref{eq:PDEint}, let us introduce the following simplified notation
\begin{equation}
G(t,x,y):=F\left[t,x,\left(\partial_{x}^{\alpha}\partial_{t}^{\gamma}y\right)_{\substack{|\alpha|\leq L\\
\gamma\le p
}
}(t,x)\right]\in\R^{m},\label{eq:Def_G}
\end{equation}
for all $(t,x)\in T\times S$ and all $y\in\mC_{t}^{p}\mC_{x}^{\infty}(T\times S,\R^{m})$.
Explicitly note that the smooth function $G(t,x,y)$ is given by composition
of $F\left[t,x,\left(z^{\alpha,\gamma}\right)_{\substack{|\alpha|\leq L\\
\gamma\le p
}
}\right]$ with the derivatives $\left(\partial_{x}^{\alpha}\partial_{t}^{\gamma}y\right)(x,t)=z^{\alpha,\gamma}\in\R^{m}$
that actually appear in \eqref{eq:PDE}. On the contrary, when we
use the variables $G(t,x,z)$, we mean that $z=\left(z^{\alpha,\gamma}\right)_{\substack{|\alpha|\leq L\\
\gamma\le p
}
}\in\R^{m\cdot\hat{L}}$.

We now introduce the following definition of Lipschitz map:
\begin{defn}
\label{def:Lip}Let $B\subseteq\mC_{t}^{p}\mC_{x}^{\infty}\left(T\times S,\R^{m}\right)$.
We say that a map $G:T\times S\times B\rightarrow\R^{m}$ is \emph{Lipschitz
on} $B$ \emph{with loss of derivatives }(LOD) $L$ \emph{and Lipschitz
factors} $\left(\Lambda_{k}\right)_{k\in\N}$ if
\begin{enumerate}
\item \label{enu:DefLipReg}$\forall y\in B:\ G\left(-,-,y\right)\in\mC_{t}^{p}\mC_{x}^{\infty}\left(T\times S,\R^{m}\right)$;
\item \label{enu:DefLipLambda}$\Lambda_{k}\in\mC^{0}(T\times S)$ for all
$k\in\N;$
\item \label{enu:DefLip}If $k\in\N$, $\nu\in\N^{s}$, $|\nu|\le k$, $h=1,\ldots,m$,
$u$, $v\in B$, $(t,x)\in T\times S$, then
\begin{equation}
\left|\partial_{x}^{\nu}G^{h}(t,x,u)-\partial_{x}^{\nu}G^{h}(t,x,v)\right|\le\Lambda_{k}(t,x)\cdot\max_{l=1,\ldots,m}\max_{\substack{|\alpha|\le k+L\\
\gamma\le p
}
}\left|\partial_{x}^{\alpha}\partial_{t}^{\gamma}(u^{l}-v^{l})(t,x)\right|.\label{eq:Lip}
\end{equation}
\end{enumerate}
We simply say that $G$ \emph{is Lipschitz on} $B$ \emph{with LOD}
$L$ if the previous conditions \ref{enu:DefLipLambda} and \ref{enu:DefLip}
hold for some $\left(\Lambda_{k}\right)_{k\in\N}$.
\end{defn}

In the next theorem, we prove that if $G$ is defined by \eqref{eq:Def_G}
and all the radii $r_{k}<+\infty$, then $G$ is always Lipschitz
with respect to \emph{constant} factors $\left(\Lambda_{k}\right)_{k\in\N}$
in the space $\bar{B}_{R}(i_{0})\subseteq\mC_{t}^{p}\mC_{x}^{\infty}\left(T\times S,\R^{m}\right)$
defined in \eqref{eq:ball-y_0-R} and with loss of derivatives $L$
given, as in \eqref{eq:PDE}, by the maximum order of derivatives
in $x$ that appears in our PDE. The space $\bar{B}_{R}(i_{0})$ is
suitable for the proof of the PLT if we are also able to prove that
for these finite radii the Picard iterates $P^{n}(i_{0})\in\bar{B}_{R}(i_{0})$.
On the other hand, in Sec.~\ref{sec:Examples} we will show examples
of PDE with constant Lipschitz factors $\Lambda_{k}$ but where we
are free to also take $r_{k}\le+\infty$. In other words, the following
result is only a sufficient condition.
\begin{thm}
\label{thm:Lipschitz for PDE}Let $r_{k}\in\R_{>0}$ for all $k\in\mathbb{N}$.
Set $R:=(r_{k})_{k\in\N}$, $i_{0}$ as in \eqref{eq:i_0} and $\bar{B}_{R}(i_{0})$
as in \eqref{eq:ball-y_0-R}, i.e:
\begin{equation}
\bar{B}_{R}(i_{0}):=\left\{ u\in\mC_{t}^{p}\mC_{x}^{\infty}(T\times S,\R^{m})\mid\Vert u-i_{0}\Vert_{k}\le r_{k}\ \forall k\in\N\right\} .\label{eq:Def_Ball}
\end{equation}
Then the function $G$ defined in \eqref{eq:Def_G} is Lipschitz in
$\bar{B}_{R}(i_{0})$ with loss of derivatives $L$ and constant Lipschitz
factors.
\end{thm}

\begin{proof}
We only have to prove condition \ref{enu:DefLip} of Def.~\ref{def:Lip},
so that we consider $k\in\N$, $\alpha\in\N^{s}$, $|\nu|\le k$,
$h=1,\ldots,m$, $u$, $v\in\bar{B}_{R}(i_{0})$, $(t,x)\in T\times S$.
Note that
\begin{equation}
\partial_{x}^{\nu}G^{h}(t,x,u)=\partial_{x}^{\nu}\left\{ F^{h}\left[t,x,\left(\partial_{x}^{\alpha}\partial_{t}^{\gamma}u\right)_{\substack{|\alpha|\leq L\\
\gamma\le p
}
}\right]\right\} .\label{eq:PartG}
\end{equation}
We first prove the case $\left|\nu\right|=0$. Since $u\in\bar{B}_{R}(i_{0})$,
we have $\Vert u-i_{0}\Vert_{L+p}\le r_{L+p}$ and hence $\partial_{x}^{\alpha}\partial_{t}^{\gamma}u(t,x)\in\overline{B_{r_{L+p}}\left(\partial_{x}^{\alpha}\partial_{t}^{\gamma}i_{0}(S)\right)}=:C_{0\alpha,\gamma}\subseteq\bigcup_{\substack{|\alpha|\le L\\
\gamma\le p
}
}C_{0\alpha,\gamma}=:C_{0}\Subset\R^{m}$ for all $|\alpha|\le L$ and $\gamma\le p$ because $r_{L+p}<+\infty$.
Similarly, $\partial_{x}^{\alpha}\partial_{t}^{\gamma}v(t,x)\in C_{0}$.
Thereby, using \eqref{eq:PartG}, we have
\begin{align*}
\left|\partial_{x}^{\nu}G^{h}(t,x,u)-\partial_{x}^{\nu}G^{h}(t,x,v)\right| & \le\Vert F\Vert_{1}\cdot\max_{l=1,\ldots,m}\max_{\substack{|\alpha|\leq L\\
\gamma\le p
}
}\left|\partial_{x}^{\alpha}\partial_{t}^{\gamma}(u^{l}-v^{l})(t,x)\right|\le\\
 & \le\Vert F\Vert_{1}\cdot\max_{l=1,\ldots,m}\max_{\substack{|\alpha|\leq k+L\\
\gamma\le p
}
}\left|\partial_{x}^{\alpha}\partial_{t}^{\gamma}(u^{l}-v^{l})(t,x)\right|,
\end{align*}
where the norm $\Vert F\Vert_{1}$ is taken on $T\times S\times C_{0}^{m\hat{L}}\Subset\R^{D}$,
$D:=\text{dim}(\text{dom}(F))=1+s+m\hat{L}$. We firstly set $\tilde{\Lambda}_{k}(\nu):=\Vert F\Vert_{1}$,
and now consider the case $|\nu|>0$.

From Faà di Bruno's formula
\begin{equation}
\partial_{x}^{\nu}G^{h}(t,x,u)=\sum_{1\le|\eta|\le|\nu|}\partial^{\eta}F^{h}\left[t,x,\left(\partial_{x}^{\alpha}\partial_{t}^{\gamma}u\right)_{\substack{|\alpha|\leq L\\
\gamma\le p
}
}\right]\cdot B_{\eta\nu}\left[\left(\partial_{x}^{\mu}\partial_{t}^{\gamma}u(t,x)\right)_{\mu\gamma}\right],\label{eq:partG^h}
\end{equation}
where $B_{\eta\beta}((z_{\mu\gamma})_{\mu\gamma})$ are Bell's like
polynomials such that $|\mu|\le|\nu|+|\alpha|\le k+L$ for all $\mu$.
For simplicity, set
\begin{align}
\partial^{\eta}F^{h}\left(t,x,u\right) & :=\partial^{\eta}F^{h}\left[t,x,\left(\partial_{x}^{\alpha}\partial_{t}^{\gamma}u\right)_{\substack{|\alpha|\leq L\\
\gamma\le p
}
}\right]\label{eq:partG}\\
Q_{\eta\nu}(t,x,u) & :=B_{\eta\nu}\left[\left(\partial_{x}^{\mu}\partial_{t}^{\gamma}u(t,x)\right)_{\mu\gamma}\right],\label{eq:Q}
\end{align}
so that we can estimate
\begin{multline*}
\left|\partial_{x}^{\nu}G^{h}(t,x,u)-\partial_{x}^{\nu}G^{h}(t,x,v)\right|=\\
=\left|\sum_{\eta}\partial^{\eta}F^{h}(t,x,u)\cdot Q_{\eta\nu}(t,x,u)-\sum_{\eta}\partial^{\eta}F^{h}(t,x,v)\cdot Q_{\eta\nu}(t,x,v)\right|\le\\
\le\left|\sum_{\eta}\partial^{\eta}F^{h}(t,x,u)\cdot Q_{\eta\nu}(t,x,u)-\sum_{\eta}\partial^{\eta}F^{h}(t,x,u)\cdot Q_{\eta\nu}(t,x,v)\right|+\\
+\left|\sum_{\eta}\partial^{\eta}F^{h}(t,x,u)\cdot Q_{\eta\nu}(t,x,v)-\sum_{\eta}\partial^{\eta}F^{h}(t,x,v)\cdot Q_{\eta\nu}(t,x,v)\right|.\qquad
\end{multline*}
For some $L_{\eta\nu}>0$ depending on $Q_{\eta\nu}$, the first summand
yields
\begin{multline*}
\left|\partial^{\eta}F^{h}(t,x,u)\right|\cdot\left|Q_{\eta\nu}(t,x,u)-Q_{\eta\nu}(t,x,v)\right|\le\\
\le\Vert F\Vert_{k}\cdot L_{\eta\nu}\cdot\max_{l=1,\ldots,m}\max_{|\mu|\le k+L}\left|\partial_{x}^{\mu}\partial_{t}^{\gamma}(u^{l}-v^{l})(t,x)\right|.
\end{multline*}
The second summand gives
\begin{multline*}
\left|\partial^{\eta}F^{h}(t,x,u)-\partial^{\eta}F^{h}(t,x,v)\right|\cdot\left|Q_{\eta\nu}(t,x,v)\right|\le\\
\le\Vert F\Vert_{k+1}\cdot\max_{l=1,\ldots,m}\max_{|\alpha|\le k+L}\left|\partial_{x}^{\alpha}\partial_{t}^{\gamma}(u^{l}-v^{l})(t,x)\right|\cdot N_{k},
\end{multline*}
where $N_{k}:=\max_{|\eta|\le|\nu|\le k}\max_{(t,x,v)\in T\times S\times C_{\eta\nu}^{k}}\left|Q_{\eta\nu}(t,x,v)\right|$
and, as we did above, $v\in\bar{B}_{R}(i_{0})$ yields some $C_{\eta\nu}^{k}\Subset\R^{m}$
such that $\partial_{x}^{\alpha}\partial_{t}^{\gamma}v(t,x)\in C_{\eta\nu}^{k}$
for all $\alpha\in\N^{s}$ such that $|\alpha|\le k+L$. We finally
obtain

\begin{align*}
\left|\partial_{x}^{\nu}G^{h}(t,x,u)-\partial_{x}^{\nu}G^{h}(t,x,v)\right| & \le\sum_{1\le|\eta|\le|\nu|}\left(\Vert F\Vert_{k}\cdot\max_{|\eta|\le|\nu|\le k}L_{\eta\nu}+\Vert F\Vert_{k+1}\cdot N_{k}\right)\cdot\\
 & \phantom{\le}\cdot\max_{l=1,\ldots,m}\max_{|\alpha|\le k+L}\left|\partial_{x}^{\alpha}\partial_{t}^{\gamma}(u^{l}-v^{l})(t,x)\right|.
\end{align*}
Setting
\begin{align*}
\tilde{\Lambda}_{k}(\nu) & :=\sum_{1\le|\eta|\le|\nu|}\left(\Vert F\Vert_{k}\cdot\max_{|\eta|\le|\nu|\le k}L_{\eta\nu}+\Vert F\Vert_{k+1}\cdot N_{k}\right)\\
\Lambda_{k} & :=\max_{|\nu|\le k}\tilde{\Lambda}_{k}(\nu),
\end{align*}
we get the conclusion.
\end{proof}

\subsection{The Picard-Lindelöf theorem for smooth normal PDE}

A natural meth\-od to solve PDE is to transform it into an infinite-dimensional
ODE and then apply a PLT, see e.g.~\cite{SeTrWa}. On the other hand,
our approach can be considered simpler because we do not transform
partial derivatives into ordinary ones in infinite dimensional spaces.

We can now state our main local existence result for smooth normal
systems of PDE.
\begin{thm}
\label{thm:PLpde}Let $\Lambda_{k}\in\mC^{0}(T\times S)$, $r_{k}\in\R_{>0}\cup\{+\infty\}$
for all $k\in\mathbb{N}$, and assume that $\mathring{S}$ is dense
in $S$. Define $R:=(r_{k})_{k\in\N}$, $\bar{B}_{R}(i_{o})$ as in
\eqref{eq:Def_Ball}, and $P:\bar{B}_{R}(i_{0})\rightarrow\mC_{t}^{p}\mC_{x}^{\infty}\left(T\times S,\R^{m}\right)$
by 
\[
P(y)(t,x):=i_{0}(t,x)+\int_{t_{0}}^{t}\diff s_{d}\ptind^{d}\int_{t_{0}}^{s_{2}}G(s_{1},x,y)\,\diff{s_{1}}.
\]
Assume that $G$ is Lipschitz on $\bar{B}_{R}(i_{0})$ with loss of
derivatives $L$ and Lipschitz factors $\left(\Lambda_{k}\right)_{k\in\N}$,
and for all $(t,x)\in T\times S$, all $n\in\N$ and all $j=1,\ldots,d$,
set
\begin{align}
\Lambda_{k,0}^{j} & :=1,\nonumber \\
\Lambda_{k,n+1}^{j}(t,x) & :=\left|\int_{t_{0}}^{t}\diff s_{j}\ptind^{j}\int_{t_{0}}^{s_{2}}\Lambda_{k}(s_{1},x)\cdot\max_{0<l\le d}\Lambda_{k+L,n}^{l}(s_{1},x)\,\diff{s_{1}}\right|,\label{eq:DefLambda}\\
\bar{\Lambda}_{k,n} & :=\max_{\substack{x\in S\\
0<j\le d
}
}\Lambda_{k,n}^{j}(t_{0}+\max(a,b),x).\nonumber 
\end{align}
Finally, assume that the following conditions are fulfilled for all
$k\in\N$:
\begin{enumerate}
\item \label{enu:iterationsPLT}$P^{n}(i_{0})\in\bar{B}_{R}(i_{0})$ for
all $n\in\N$;
\item \label{enu:W}${\displaystyle \sum_{n=0}^{+\infty}}\bar{\Lambda}_{k,n}\cdot\left\Vert P\left(i_{0}\right)-i_{0}\right\Vert _{k+nL}<+\infty$.
\end{enumerate}
Then, there exists a smooth solution $y\in\bar{B}_{R}(i_{0})\cap\mC^{\infty}(T\times S,\R^{m})$
of the problem
\begin{equation}
\begin{cases}
\partial_{t}^{d}y(t,x)=F\left[t,x,\left(\partial_{x}^{\alpha}\partial_{t}^{\gamma}y\right)_{\substack{|\alpha|\leq L\\
\gamma\le p
}
}\right],\\
\partial_{t}^{j}y(t_{0},x)=y_{0j}(x)\ j=0,\ldots,d-1,
\end{cases}\label{eq:PDEred_y}
\end{equation}
given by $y=\lim_{n\rightarrow+\infty}P^{n}\left(i_{0}\right)$ in
$\left(\mC_{t}^{p}\mC_{x}^{\infty}\left(T\times S,\R^{m}\right),\left(\left\Vert -\right\Vert _{k}\right)_{k\in\N}\right)$,
which satisfies
\[
\forall k,m\in\N:\ \left\Vert y-P^{m}\left(i_{0}\right)\right\Vert _{k}\leq\sum_{n=m}^{+\infty}\bar{\Lambda}_{k,n}\cdot\left\Vert P\left(i_{0}\right)-i_{0}\right\Vert _{k+nL}.
\]
In particular, if $M_{k}\in\mC^{0}(T\times S)$, we set
\[
\bar{M}_{kj}(t,x):=\left|\int_{t_{0}}^{t}\diff s_{j}\ptind^{j}\int_{t_{0}}^{s_{2}}M_{k}(s_{1},x)\,\diff{s_{1}}\right|,
\]
and we also assume
\begin{enumerate}[resume]
\item \label{enu:GM_i}$\left|\partial_{x}^{\nu}G(t,x,u)\right|\le M_{k}(t,x)$
for all $u\in\bar{B}_{R}(i_{0})$, $(t,x)\in T\times S$ and all $\nu\in\N^{s}$
such that $|\nu|\le k$;
\item \label{enu:rad}${\displaystyle \max_{\substack{x\in S\\
0<j\le d
}
}}\bar{M}_{k,j}(t_{0}+\max(a,b),x)\le r_{k}$;
\end{enumerate}
Then $P:\bar{B}_{R}(i_{0})\ra\bar{B}_{R}(i_{0})$ and hence \ref{enu:iterationsPLT}
always holds.
\end{thm}

\begin{proof}
We prove that $P$ actually satisfies the stronger contraction property
\eqref{eq:classContrLoD} with contraction constants $\bar{\Lambda}_{kn}$.
We firstly show, by induction on $n\in\N$, that for each $k\in\N$,
$u$, $v\in\bar{B}_{R}(i_{0})$, $(t,x)\in T\times S$, $h=1,\ldots,m$,
and $\beta\in\N_{p}^{1+s}$ with $|\beta|\le k$, we have
\begin{equation}
\left|\partial^{\beta}\left[P^{n}(u)^{h}-P^{n}(v)^{h}\right](t,x)\right|\le\Vert u-v\Vert_{k+nL}\cdot\max_{0<j\le d}\Lambda_{kn}^{j}(t,x).\label{eq:HpInd}
\end{equation}
For $n=0$, \eqref{eq:HpInd} reduces to $\left|\partial^{\beta}(u^{h}-v^{h})(t,x)\right|\le\Vert u-v\Vert_{k}\cdot\max_{0<j\le d}\Lambda_{k,0}^{j}(t,x)$
which holds because $|\beta|\le k$ and $\Lambda_{k,0}^{j}=1$. To
prove the inductive step, we consider
\begin{multline}
\left|\partial^{\beta}\left[P^{n+1}(u)^{h}-P^{n+1}(v)^{h}\right](t,x)\right|\le\left|\partial^{\beta}\left\{ {\displaystyle \int_{t_{0}}^{t}}\diff s_{d}\ptindShort^{d}\int_{t_{0}}^{s_{2}}G^{h}(s_{1},x,P^{n}(u))\,\diff{s_{1}}\right.\right.\\
\left.\left.-\int_{t_{0}}^{t}\diff s_{d}\ptindShort^{d}\int_{t_{0}}^{s_{2}}G^{h}(s_{1},x,P^{n}(v))\,\diff{s_{1}}\right\} \right|=:(1^{*})\label{eq:n-n+1}
\end{multline}
Since $\beta\in\N_{p}^{1+s}$, we can write $\partial^{\beta}=\partial_{x}^{\nu}\partial_{t}^{\beta_{1}}$,
where $\nu:=(\beta_{2},\dots,\beta_{s})$ and $\beta_{1}\le p<d$.
The operator $\partial_{t}^{\beta_{1}}$ deletes $\beta_{1}$ integrals
in \eqref{eq:n-n+1}; set $\bar{\jmath}:=d-\beta_{1}>0$, and take
$\partial_{x}^{\nu}$ inside the integrals to get
\begin{align*}
(1^{*}) & \le\text{sgn}(t-t_{0})^{\bar{\jmath}}\int_{t_{0}}^{t}\diff s_{\bar{\jmath}}\ptindShort^{\bar{\jmath}}\int_{t_{0}}^{s_{2}}\left|\partial_{x}^{\nu}G^{h}(s_{1},x,P^{n}(u))-\partial_{x}^{\nu}G^{h}(s_{1},x,P^{n}(v))\right|\,\diff{s_{1}}\\
 & =:(2^{*}).
\end{align*}
Since $G$ is Lipschitz on $\bar{B}_{R}(i_{0})$ with factors $\left(\Lambda_{k}\right)_{k\in\N}$,
we get
\begin{multline*}
(2^{*})\le\text{sgn}(t-t_{0})^{\bar{\jmath}}\int_{t_{0}}^{t}\diff s_{\bar{\jmath}}\ptindShort^{\bar{\jmath}}\int_{t_{0}}^{s_{2}}\Lambda_{k}(s_{1},x)\cdot\\
\cdot\max_{l=1,\ldots,m}\max_{\substack{|\alpha|\le k+L\\
\gamma\le p
}
}\left|\partial_{x}^{\alpha}\partial_{t}^{\gamma}\left[P^{n}(u)^{l}-P^{n}(v)^{l}\right](s_{1},x)\right|\,\diff{s_{1}}=:(3^{*}).
\end{multline*}
Using inductive hypothesis \eqref{eq:HpInd} (with $k+L$ instead
of $k$ and $(\gamma,\alpha)$ instead of $\beta$)
\begin{align*}
(3^{*}) & \le\text{sgn}(t-t_{0})^{\bar{\jmath}}\int_{t_{0}}^{t}\diff s_{\bar{\jmath}}\ptindShort^{\bar{\jmath}}\int_{t_{0}}^{s_{2}}\Lambda_{k}(s_{1},x)\cdot\Vert u-v\Vert_{k+L+nL}\cdot\\
 & \phantom{\le}\cdot\max_{0<l\le d}\Lambda_{k+L,n}^{l}(s_{1},x)\,\diff{s_{1}}\\
 & =\Vert u-v\Vert_{k+(n+1)L}\cdot\left|\int_{t_{0}}^{t}\diff s_{\bar{\jmath}}\ptindShort^{\bar{\jmath}}\int_{t_{0}}^{s_{2}}\Lambda_{k}(s_{1},x)\cdot\right.\\
 & \phantom{=}\left.\cdot\max_{0<l\le d}\Lambda_{k+L,n}^{l}(s_{1},x)\,\diff{s_{1}}\right|\\
 & =\Vert u-v\Vert_{k+(n+1)L}\cdot\Lambda_{k,n+1}^{\bar{\jmath}}(t,x)\\
 & \le\Vert u-v\Vert_{k+(n+1)L}\cdot\max_{0<j\le d}\Lambda_{k,n+1}^{j}(t,x),
\end{align*}
which proves our claim.

Finally, we prove \eqref{eq:classContrLoD}: for some $\beta\in\N_{p}^{1+s}$,
$|\beta|\le k$, some $h=1,\ldots,m$ and some $(t,x)\in T\times S$,
from \eqref{eq:HpInd} we have
\begin{align*}
\Vert P^{n}(u)-P^{n}(v)\Vert_{k} & =\left|\partial^{\beta}\left[P^{n}(u)^{h}-P^{n}(v)^{h}\right](t,x)\right|\le\\
 & \le\Vert u-v\Vert_{k+nL}\cdot\max_{0<j\le d}\Lambda_{kn}^{j}(t,x)\le\\
 & \le\Vert u-v\Vert_{k+nL}\cdot\bar{\Lambda}_{kn}.
\end{align*}
This shows the claim on $P$ with contraction constants $\bar{\Lambda}_{kn}$.
The conclusion with $y\in\bar{B}_{R}(i_{0})$ hence follows from Weissinger
condition \ref{enu:W} and Thm.~\ref{thm:BFPTLoss}. It only remains
to prove that actually $y$ is smooth. Since $y$ is a fixed point
of $P$, we have
\begin{align}
y(t,x) & =i_{0}(t,x)+\int_{t_{0}}^{t}\diff s_{d}\ptind^{d}\int_{t_{0}}^{s_{2}}G(s_{1},x,y)\,\diff{s_{1}}\label{eq:fixed}\\
 & =i_{0}(t,x)+\int_{t_{0}}^{t}\diff s_{d}\ptind^{d}\int_{t_{0}}^{s_{2}}F\left[s_{1},x,\left(\partial_{x}^{\alpha}\partial_{t}^{\gamma}y\right)_{\substack{|\alpha|\leq L\\
\gamma\le p
}
}(s_{1},x)\right]\,\diff{s_{1}}.\nonumber 
\end{align}
But $y\in\mC_{t}^{p}\mC_{x}^{\infty}(T\times S,\R^{m})$ and hence
$\left(\partial_{x}^{\alpha}\partial_{t}^{\gamma}y\right)_{\substack{|\alpha|\leq L\\
\gamma\le p
}
}\in\mathcal{C}^{0}(T\times S,\R^{m})$. By induction \eqref{eq:fixed} proves that $y$ is smooth at interior
points of $T\times S$ and hence also at boundary points by continuity
of derivatives on $\mathring{T}\times\mathring{S}$.

In particular, if we assume both \ref{enu:GM_i} and \ref{enu:rad},
we can prove that $P:\bar{B}_{R}(i_{0})\ra\bar{B}_{R}(i_{0})$ using
Lem.~\ref{lem:normsInt}. In fact, for any $u\in\bar{B}_{R}(i_{0})$
and $k\in\N$, from \ref{enu:GM_i} and \ref{enu:rad} we have
\begin{align*}
\left\Vert P(u)-i_{0}\right\Vert _{k}= & \left\Vert \int_{t_{0}}^{(-)}\diff s_{d}\ptind^{d}\int_{t_{0}}^{s_{2}}G\left(s_{1},-,u\right)\,\diff{s_{1}}\right\Vert _{k}\\
\le & \max_{\substack{x\in S\\
0<j\le d
}
}\bar{M}_{k}(t_{0}+\max(a,b),x)\le r_{k}.
\end{align*}
\end{proof}
Note that, on the contrary with respect to the more classical conditions
\ref{enu:GM_i} and \ref{enu:rad} (inherited from the classical PLT
for ODE) depending both on the choice of upper bounds $M_{k}$ and
radii $r_{k}$, assumption \ref{enu:iterationsPLT} depends only on
the radii $r_{k}$. In Sec.~\ref{sec:Examples}, we will see that
requirement \ref{enu:iterationsPLT} leads us to the correct choice
of these radii $r_{k}$ (one more time, depending on the initial conditions
$r_{k}=r_{k}(i_{0})$).

If the radii $r_{k}<+\infty$, we can also consider as bounds $M_{k}$
of \ref{enu:GM_i} and \ref{enu:rad} the minimal constant functions.
This is considered in the following result, which allows us to understand
that to have a local solution using the latter part of the PLT, we
have to avoid that $\frac{r_{k}}{M_{k}}\to0$:
\begin{cor}
\label{cor:MiConst}Let $\Lambda_{k}\in\mC^{0}(T\times S)$, $r_{k}\in\R_{>0}$
for all $k\in\mathbb{N}$. Assume that $\mathring{S}$ is dense in
S and $G$ is Lipschitz on $\bar{B}_{R}(i_{0})$ with loss of derivatives
$L$ and Lipschitz factors $\left(\Lambda_{k}\right)_{k\in\N}$, define
$\bar{\Lambda}_{kn}$ as in \eqref{eq:DefLambda} and
\begin{align*}
C_{k} & :=T\times S\times\bigcup_{\substack{|\nu|\le k+L\\
\gamma\le p
}
}\overline{B_{r_{k+L+p}}\left(\partial_{x}^{\nu}\partial_{t}^{\gamma}i_{0}(T\times S)\right)}\\
M_{k} & :=\max_{(t,x,z)\in C_{k}}\max_{|\nu|\le k}\left|\partial_{x}^{\nu}G(t,x,z)\right|
\end{align*}
Finally, assume that the following conditions are fulfilled:
\begin{enumerate}
\item \label{enu:radConst}$\max(a,b)\le\inf_{k\in\N}\frac{r_{k}}{M_{k}}$;
\item \label{enu:.condizione (W)Repl}$\sum_{n=0}^{+\infty}\bar{\Lambda}_{kn}\cdot\left\Vert P\left(i_{0}\right)-i_{0}\right\Vert _{k+nL}<+\infty$
for all $k\in\N$.
\end{enumerate}
Then, there exists a smooth solution $y\in\bar{B}_{R}(i_{0})\cap\mC^{\infty}(T\times S,\R^{m})$
of problem \eqref{eq:PDEred_y}.

\end{cor}

\begin{proof}
Note explicitly that $C_{k}\Subset\R^{1+s+m}$ because $r_{k+L+p}<+\infty$.
As we proved in Thm.~\ref{thm:Lipschitz for PDE}, for all $u\in\bar{B}_{R}(i_{0})$,
all $|\nu|\le k+L$ and all $\gamma\le p$, we have
\[
\left|\partial_{x}^{\nu}\partial_{t}^{\gamma}u(t,x)-\partial_{x}^{\nu}\partial_{t}^{\gamma}i_{0}(t,x)\right|\le\|u-i_{0}\|_{|\nu|+\gamma}\le\|u-i_{0}\|_{k+L+p}\le r_{k+L+p},
\]
and hence $\partial_{x}^{\nu}\partial_{t}^{\gamma}u(t,x)\in C_{k}$.
Thereby, condition Thm.~\ref{thm:PLpde}.\ref{enu:GM_i} holds for
the chosen constant $M_{k}$ (see also Rem.~\ref{rem:PLT} just below).
Therefore, $\bar{M}_{kj}(t,x)=\frac{(t-t_{0})^{j}}{j!}M_{k}$ and
$\max_{\substack{x\in S\\
0<j\le d
}
}\bar{M}_{k}(t_{0}+\max(a,b),x)=\max(a,b)\cdot M_{k}\le r_{k}$ for all $k\in\N$ by \ref{enu:radConst}. We can finally apply Thm\@.~\ref{thm:PLpde}.
\end{proof}
\begin{rem}
\label{rem:PLT}To avoid misunderstandings, we explicitly note that
the simplified notation $\partial_{x}^{\nu}G(t,x,z)$ denotes the
function obtained by the following process:
\begin{enumerate}
\item Consider and arbitrary $u\in\bar{B}_{R}(i_{0})$, and the derivative
$\partial_{x}^{\nu}\left(G(t,x,u)\right)(t,x)$ given by \eqref{eq:partG^h};
\item In the formula \eqref{eq:partG^h} obtained after the computation
of this derivative, sub\-sti\-tute the variables $z_{\alpha\gamma}:=\left(\partial_{x}^{\alpha}\partial_{t}^{\gamma}u\right)_{\substack{|\alpha|\leq L\\
\gamma\le p
}
}$ and $z_{\mu\gamma}:=\left(\partial_{x}^{\mu}\partial_{t}^{\gamma}u(t,x)\right)_{\mu\gamma}$
to obtain $\partial_{x}^{\nu}G(t,x,z)$, where the variable $z$ represents
all the $z_{\alpha\gamma}$ and $z_{\mu\gamma}$.
\end{enumerate}
For example, for the PDE $\partial_{t}^{d}y=a(t)\cdot\partial_{x}^{L}\partial_{t}^{\gamma}y$,
we calculate the derivatives as $\partial_{x}^{\nu}G(t,x,u)=a(t)\cdot\partial_{x}^{\nu+L}\partial_{t}^{\gamma}u(t,x)$,
and here substituting $z$ for $\partial_{x}^{\nu+L}\partial_{t}^{\gamma}u(t,x)$
we get $\partial_{x}^{\nu}G(t,x,z)=a(t)\cdot z$. For the PDE $\partial_{t}^{d}y=\left(\partial_{x}^{L}\partial_{t}^{\gamma}y\right)^{2}$,
we have e.g.~$\partial_{x}G(t,x,u)=2\partial_{x}^{L}\partial_{t}^{\gamma}u(t,x)\partial_{x}^{L+1}\partial_{t}^{\gamma}u(t,x)$,
and $\partial_{x}G(t,x,z_{0},z_{1})=2z_{0}z_{1}$.
\end{rem}

On the contrary with respect to the Cauchy-Kowalevski theorem, in
the PL Thm.~\ref{thm:PLpde}, it would appear that we do not need
to assume $d\ge L$. However, this clearly cannot hold in general,
and in Sec.~\ref{sec:Examples} we show that such type of assumption
is implicitly contained in the convergence request of Weissinger condition
Thm.~\ref{thm:PLpde}.\ref{enu:W}. A first partial confirmation
going in this direction, can be glimpsed by computing the iteration
$P^{n}(i_{0})(t,x)$ in case of simple linear PDE, and then taking
$n\to+\infty$:
\begin{example}
\label{exa:Aleksandr}Assuming all the needed hypotheses to apply
the PL Thm.~\ref{thm:PLpde} for each one of the following cases
(where $a\in\R_{\ne0}$), we have:
\begin{enumerate}
\item \label{enu:heat}If $\partial_{t}y=a\cdot\partial_{x}^{2}y$, then
$y(t,x)=\sum_{n=0}^{+\infty}\frac{\partial_{x}^{2n}y_{00}(x)}{n!}a^{n}\cdot(t-t_{0})^{n}$;
\item \label{enu:waveLap}If $\partial_{t}^{2}y=a\cdot\partial_{x}^{2}y$,
then $y(t,x)=\sum_{n=0}^{+\infty}\frac{\partial_{x}^{2n}y_{00}(x)}{n!}a^{n}(t-t_{0})^{n}+\sum_{n=0}^{+\infty}\frac{\partial_{x}^{2n}y_{01}(x)}{(n+1)!}a^{n}(t-t_{0})^{n+1}$;
\item If $\partial_{t}y=a\cdot\partial_{x}y$, then $y(t,x)=\sum_{n=0}^{+\infty}\frac{\partial_{x}^{n}y_{00}(x)}{n!}a^{n}\cdot(t-t_{0})^{n}$;
\item \label{enu:p<d}If $\partial_{t}^{2}y=a\cdot\partial_{t}\partial_{x}y$,
then $y(t,x)=y_{0}(x)+\sum_{n=0}^{+\infty}\frac{\partial_{x}^{n}y_{01}(x)}{(n+1)!}a^{n}(t-t_{0})^{n+1}$;
\item If $\partial_{t}^{2}y=a\cdot\partial_{x}y$, then $y(t,x)=\sum_{n=0}^{+\infty}\frac{\partial_{x}^{n}y_{00}(x)}{(2n)!}a^{n}(t-t_{0})^{2n}+\sum_{n=0}^{+\infty}\frac{\partial_{x}^{n}y_{01}(x)}{(2n+1)!}a^{n}(t-t_{0})^{2n+1}$.
\end{enumerate}
\end{example}

The ideas of the proof of Thm.~\ref{thm:PLpde} are a simple generalization
of the classical proof for ODE, only adapted to contractions with
LOD and a countable family of norms. Indeed, for $L=0$ and $\Vert-\Vert_{k}=\Vert-\Vert_{0}$
the proof reduces to the classical proof for ODE and assumptions \ref{enu:GM_i},
\ref{enu:rad}, \ref{enu:.condizione (W)Repl} reduce to the usual
ones for the PLT for ODE with Weissinger condition, see e.g.~\cite{Tes13}.
On the other hand, the compact set $S\Subset\R^{s}$ (with $\mathring{S}$
dense in S) is \emph{completely arbitrary}: we can hence say that
our deduction proves that, \emph{with respect to the PLT}, PDE can
be simply treated as ODE depending on a parameter $x\in S$.

If the Lipschitz factors $\Lambda_{k}\in\R$ and the upper bounds
$M_{k}\in\R$ are constant (the proof of Thm\@.~\ref{thm:Lipschitz for PDE}
and Cor.~\ref{cor:MiConst} show that this is not a loss of generality),
then
\begin{align*}
\bar{M}_{kj}(t,x) & =M_{k}\frac{|t-t_{0}|^{j}}{j!},\\
\Lambda_{kn}^{j}(t,x) & =\frac{|t-t_{0}|^{nd+j}}{(nd+j)!}\prod_{j=0}^{n-1}\Lambda_{k+jL},\\
\bar{\Lambda}_{kn} & =\frac{\max(a,b)^{nd}}{(nd)!}\prod_{j=0}^{n-1}\Lambda_{k+jL}.
\end{align*}
Thereby, Weissinger condition Thm.~\ref{thm:PLpde}.\ref{enu:W}
becomes
\begin{equation}
\sum_{n=0}^{+\infty}\frac{\max(a,b)^{nd}}{(nd)!}\Vert P(i_{0})-i_{0}\Vert_{k+nL}\prod_{j=0}^{n-1}\Lambda_{k+jL}<+\infty\qquad\forall k\in\N.\label{eq:WConst}
\end{equation}
For ODE, we have $L=0$ and $\|-\|_{k}=\|-\|_{0}$, and \eqref{eq:WConst}
reduces to
\[
\|P(i_{0})-i_{0}\|_{0}\sum_{n=0}^{+\infty}\Lambda_{0}^{n}\frac{\max(a,b)^{nd}}{(nd)!}<+\infty,
\]
which always holds.
\begin{rem}
\ 
\begin{enumerate}
\item We believe it is worth mentioning that in a non-Archimedean setting
such as that of generalized smooth functions theory and Robinson-Colombeau
ring $\rti$, see e.g.~\cite{GiKuVe15,GiKuVe21,LuGi22,GiLu22}, we
can repeat the proof of the PLT with exactly the same formal steps
(but with the ring $\rti$ instead of the field $\R$). In addition,
we can take $S=[-\iota,\iota]^{s}\supseteq\R^{s}$, where $\iota$
is an infinite number (this kind of sets behave as compact sets for
generalized smooth functions, see \cite{GiKu18}). In this way, we
get a global solution in $x\in S$ (but clearly, we need initial conditions
on $S$, or equivalently boundary conditions that hold for all $x\in\R$).
Moreover, in the setting of the non-Archimedean ring $\rti$, the
generalized number (equivalence class in the quotient ring $\rti$)
$\diff\rho:=[\rho_{\eps}]\in\rti$ is an infinitesimal number since
$\rho_{\eps}\to0^{+}$ as $\eps\to0^{+}$ (and hence $\diff\rho^{-Q}=\left[\rho_{\eps}^{-Q}\right]$
is an infinite number for all $Q\in\N_{>0}$). Since in every Cauchy
complete non-Archimedean ring a series converges if and only if the
general term tends to zero (see e.g.~\cite{Kob96}), condition \eqref{eq:WConst}
is equivalent to $\lim_{n\to+\infty}\frac{\max(a,b)^{nd}}{(nd)!}\Vert P(i_{0})-i_{0}\Vert_{k+nL}\prod_{j=0}^{n-1}\Lambda_{k+jL}=0$.
Assuming that for some $Q\in\N_{\ge0}$ and some $q\in\N_{>0}$, we
have
\begin{align}
\frac{\max(a,b)^{nd}}{(nd)!}\Vert P(i_{0})-i_{0}\Vert_{k+nL}\prod_{j=0}^{n-1}\Lambda_{k+jL} & \le\diff\rho^{-Q},\label{eq:infinite}\\
\max(a,b) & \le\diff\rho^{q},\nonumber 
\end{align}
then $\frac{\max(a,b)^{nd}}{(nd)!}\Vert P(i_{0})-i_{0}\Vert_{k+nL}\prod_{j=0}^{n-1}\Lambda_{k+jL}\le\diff\rho^{-Q}\cdot\frac{\text{{\rm d}}\rho{}^{qnd}}{(nd)!}\to0$
as $n\to+\infty$. Thereby, since for ordinary smooth functions the
left hand side of \eqref{eq:infinite} is finite, we have that \emph{any
ordinary smooth normal Cauchy problem (even Lewy-Mizohata examples)
always has a solution in an infinitesimal interval} (see \cite{GiLu22}
for greater details).
\item Lewy-Mizohata examples imply that Weissinger condition in these cases
does not hold. Moreover, since the non-existence of a solution does
not depend on the initial condition $i_{0}$, \cite{Lew57,Miz62},
and taking $a=-1$, $b=1$, we necessarily must have
\begin{equation}
\exists k\in\N:\ \sum_{n=0}^{+\infty}\frac{1}{n!}\prod_{j=0}^{n-1}\Lambda_{k+j}=+\infty\label{eq:WLM}
\end{equation}
for all Lipschitz factors $(\Lambda_{j})_{j\in\N}$ (that always exist
because of Thm.~\ref{thm:Lipschitz for PDE}) (note that $d=1=L$,
$m=2$ for both counter-examples). Condition \eqref{eq:WLM} strongly
recall the non-analytic nature of $F$ in these cases.
\end{enumerate}
\end{rem}

\section{\label{sec:Examples}Examples}

The main aim of this section is to show at least one example of normal
PDE \eqref{eq:PDE} where either the right hand side $F$ or one of
the initial conditions $y_{0j}$ are not analytic functions.

The class of examples we are going to consider is
\begin{equation}
\partial_{t}^{d}y(t,x)=p(t)\cdot\partial_{x}^{\mu}\partial_{t}^{\gamma}y(t,x)+q(t,x),\label{eq:ex}
\end{equation}
where $y(t,x)\in\R^{m}$, $\mu\in\N^{s}$, $|\mu|=L>0$, $0\le\gamma<d$,
$p\in\Coo(T,\R^{m\times m})$ is an arbitrary matrix valued smooth
function, and $q$ is a smooth function with uniformly bounded derivatives
in $x$:
\begin{align}
\exists Q & \in\R_{>0}\,\forall\nu\in\N^{s}\,\forall(t,x)\in T\times S:\ |\partial_{x}^{\nu}q(t,x)|\le Q.\label{eq:qQ}
\end{align}
Clearly, wave, heat and Laplace equations are particular cases of
\eqref{eq:ex}. Explicitly note that also Mizohata's counterexample
\cite{Miz62} $\partial_{t}y^{1}=t\partial_{x}y^{2}+q^{1}(t,x)$,
$\partial_{t}y^{2}=-t\partial_{x}y^{1}+q^{2}(t,x)$ is exactly of
the form \eqref{eq:ex}, but $q=(q^{1},q^{2})$ is not analytic, so
it does not satisfy condition \eqref{eq:qQ}. On the other hand, in
Lewy's counterexample the coefficient $p=p(t,x_{1},x_{2})$ also depends
on $x$ and the term $q$ is again not analytic.

We want to directly apply the first part of the PL Thm.~\ref{thm:PLpde},
i.e.~to check properties \ref{enu:iterationsPLT} and \ref{enu:W}.
We start focusing on the latter, and arriving at estimates that are
independent from the radii $r_{k}$.

\subsection*{Estimate of Lipschitz constants $\Lambda_{k}$}

For arbitrary radii $R=(r_{k})_{k\in\N}$ and all $u$, $v\in\bar{B}_{R}(i_{0})$
and all $\nu\in\N^{s}$, with $|\nu|\le k$, we have:
\[
\left|\partial_{x}^{\nu}G(t,x,u)-\partial_{x}^{\nu}G(t,x,v)\right|\le\|p\|_{0}\cdot\max_{|\alpha|\le k+L}\left|\partial_{x}^{\alpha}u(t,x)-\partial_{x}^{\alpha}v(t,x)\right|.
\]
We can hence consider $\Lambda_{k}:=\|p\|_{0}$, so that setting $\bar{T}:=\max(a,b)$,
we get
\begin{equation}
\bar{\Lambda}_{kn}=\frac{\bar{T}^{nd}}{(nd)!}\prod_{j=0}^{n-1}\Lambda_{k+jL}=\frac{\bar{T}^{nd}}{(nd)!}\|p\|_{0}^{n}.\label{eq:LambdaBar}
\end{equation}

\subsection*{Estimate of the terms $\|P(i_{0})-i_{0}\|_{k+nL}$}

We have
\begin{align*}
P(i_{0})(t,x)-i_{0}(t,x) & =\int_{t_{0}}^{t}\diff s_{d}\ptind^{d}\int_{t_{0}}^{s_{2}}\left[p(s_{1})\cdot\partial_{x}^{\mu}\partial_{t}^{\gamma}i_{0}(s_{1},x)+q(s_{1},x)\right]\diff s_{1}\\
\partial_{x}^{\mu}\partial_{t}^{\gamma}i_{0}(s_{1},x) & =\sum_{j=\gamma}^{d-1}\frac{\partial_{x}^{\mu}y_{0j}(x)}{j!}j(j-1)\cdot\ldots\cdot(j-\gamma+1)\cdot(t-t_{0})^{j-\gamma}\\
 & =\sum_{j=\gamma}^{d-1}\frac{\partial_{x}^{\mu}y_{0j}(x)}{(j-\gamma)!}(t-t_{0})^{j-\gamma}.
\end{align*}
For $\beta\in\N^{1s}$, $|\beta|\le k+nL$, $\beta_{x}:=(\beta_{2},\ldots,\beta_{s})$,
we thus have
\begin{align*}
\partial^{\beta}\left[P(i_{0})(t,x)-i_{0}(t,x)\right] & =\sum_{j=\gamma}^{d-1}\frac{\partial_{x}^{\mu+\beta_{x}}y_{0j}(x)}{(j-\gamma)!}\int_{t_{0}}^{t}\diff s_{d}\ptind^{d-\beta_{1}}\int_{t_{0}}^{s_{2}}p(s_{1})\cdot(s_{1}-t_{0})^{j-\gamma}\,\diff s_{1}+\\
 & \phantom{=}\int_{t_{0}}^{t}\diff s_{d}\ptind^{d-\beta_{1}}\int_{t_{0}}^{s_{2}}\partial_{x}^{\beta_{x}}q(s_{1},x)\,\diff s_{1}.
\end{align*}
Using assumption \eqref{eq:qQ}:
\begin{align*}
\left|\partial^{\beta}\left[P(i_{0})(t,x)-i_{0}(t,x)\right]\right| & \le\|p\|_{0}\sum_{j=\gamma}^{d-1}\frac{\left|\partial_{x}^{\mu+\beta_{x}}y_{0j}(x)\right|}{(j-\gamma)!}\frac{\bar{T}^{j-\gamma+d-\beta_{1}}}{(j-\gamma+d-\beta_{1})!}(j-\gamma)!+\\
 & \phantom{\le}+Q\frac{\bar{T}^{d-\beta_{1}}}{(d-\beta_{1})!}.
\end{align*}
Therefore, taking for simplicity $\bar{T}\le1$:
\begin{equation}
\|P(i_{0})-i_{0}\|_{k+nL}\le\|p\|_{0}\sum_{j=\gamma}^{d-1}\frac{\|y_{0j}\|_{k+(n+1)L}}{(j-\gamma+d)!}+Q.\label{eq:normP}
\end{equation}

\subsection*{Weissinger condition}

Based on \eqref{eq:LambdaBar} and \eqref{eq:normP}, we can estimate
Weissinger condition as
\begin{multline*}
\sum_{n=0}^{+\infty}\bar{\Lambda}_{kn}\cdot\|P(i_{0})-i_{0}\|_{k+nL}\le\\
\sum_{n=0}^{+\infty}\frac{\bar{T}^{nd}}{(nd)!}\|p\|_{0}^{n+1}\sum_{j=\gamma}^{d-1}\frac{\|y_{0j}\|_{k+(n+1)L}}{(j-\gamma+d)!}+\sum_{n=0}^{+\infty}\frac{\bar{T}^{nd}}{(nd)!}\|p\|_{0}^{n}Q\\
=:S_{1}+S_{2}.
\end{multline*}
Since the latter series $S_{2}$ is convergent, we focus on the first
one:
\begin{equation}
S_{1}=\|p\|_{0}\sum_{n=0}^{+\infty}\left(\bar{T}^{d}\|p\|_{0}\right)^{n}\sum_{j=\gamma}^{d-1}\frac{\|y_{0j}\|_{k+(n+1)L}}{(nd)!}\cdot\frac{1}{(j-\gamma+d)!}.\label{eq:S_1}
\end{equation}
In this series, the only potentially problematic terms are the fractions
$\frac{\|y_{0j}\|_{k+(n+1)L}}{(nd)!}$ as $n\to+\infty$ for $j=\gamma,\ldots,d-1$,
because the remaining part surely yields a convergent series for $\bar{T}$
sufficiently small. This also yields that all the possible initial
conditions $y_{0j}$ for $j=0,\ldots,\gamma-1$ can be freely chosen
(see also Example \ref{exa:Aleksandr}.\ref{enu:p<d}).

\subsection*{Estimate of radii $r_{k}$}

If $f=f(t)$ is a function of $t$, for simplicity we first set
\[
I_{d}\left[f(t)\right]:=\int_{t_{0}}^{t}\diff s_{d}\ptind^{d}\int_{t_{0}}^{s_{2}}f(s_{1})\,\diff s_{1}.
\]
For $j\ge\gamma$ and $h\in\N$, we set
\begin{align*}
\mu_{j-\gamma,0}(t) & :=p(t)\cdot(t-t_{0})^{j-\gamma}\\
\mu_{j-\gamma,h+1}(t) & :=I_{d}\left[p(t)\cdot\partial_{t}^{\gamma}\mu_{j-\gamma,h}(t)\right]\\
\eta_{0}(t,x) & :=q(t,x)\\
\eta_{h+1}(t,x) & :=I_{d}\left[p(t)\cdot\partial_{x}^{\mu}\partial_{t}^{\gamma}\eta_{h}(t,x)\right].
\end{align*}
By induction on $n\in\N$, we can then prove that
\begin{equation}
P^{n}(i_{0})(t,x)=\sum_{j=0}^{d-1}\frac{y_{0j}(x)}{j!}(t-t_{0})^{j}+\sum_{h=1}^{n}\sum_{j=\gamma}^{d-1}\frac{\partial_{x}^{h\mu}y_{0j}(x)}{(j-\gamma)!}\mu_{j-\gamma,h}(t)+\sum_{h=1}^{n}\eta_{h}(t,x).\label{eq:P^n}
\end{equation}
Thereby, we get in this way a possible definition of the radii $r_{k}$
as
\[
\|P^{n}(i_{0})-i_{0}\|_{k}\le\sum_{h=1}^{+\infty}\sum_{j=\gamma}^{d-1}\frac{\|y_{0j}\|_{k+hL}}{(j-\gamma)!}\|\mu_{j-\gamma,h}\|_{k}+\sum_{h=0}^{+\infty}\|\eta_{h}\|_{k}=:r_{k}\quad\forall k\in\N.
\]
Using assumption \eqref{eq:qQ}, the latter series converges because
\[
\|\eta_{h}\|_{k}\le\|p\|_{0}^{h}\cdot Q\cdot\frac{\bar{T}^{(h+1)(d-\gamma)}}{\left[(h+1)(d-\gamma)\right]!}.
\]
In the former series, for $h\ge1$ we have instead
\[
\|\mu_{j-\gamma,h}\|_{k}\le\|p\|_{0}^{h}\cdot\frac{\bar{T}^{h(d-\gamma)+j}}{\left[(d-\gamma)!\right]^{h}},
\]
so that
\[
\sum_{h=1}^{+\infty}\sum_{j=\gamma}^{d-1}\frac{\|y_{0j}\|_{k+hL}}{(j-\gamma)!}\|\mu_{j-\gamma,h}\|_{k}\le\sum_{h=1}^{+\infty}\left(\frac{\|p\|_{0}\bar{T}^{d-\gamma}}{(d-\gamma)!}\right)^{h}\cdot\sum_{j=\gamma}^{d-1}\frac{\|y_{0j}\|_{k+hL}\bar{T}^{j}}{(j-\gamma)!}.
\]
If this series converges (and this mainly depends on the growing of
$\|y_{0j}\|_{k+hL}$), we can hence have $r_{k}<+\infty$, otherwise
we simply take $r_{k}=+\infty$.

\subsubsection*{A case of exponentially growing initial conditions}

If all the functions $y_{0j}$, $j=\gamma,\ldots,d-1$, satisfy for
some some $C_{j}\in\R_{>0}$ 
\begin{equation}
\|y_{0j}\|_{k+(n+1)L}\le C_{j}^{k+(n+1)L}\quad\forall j=\gamma,\ldots,d-1\,\forall k\in\N,\label{eq:bounded}
\end{equation}
then the series \eqref{eq:S_1} converges and we have
\begin{thm}
\label{thm:exp}If the initial conditions $y_{0j}$ satisfy \eqref{eq:bounded},
whereas $y_{0j}$ for $j=0,\ldots,\gamma-1$ are arbitrary smooth
functions, then there exists a smooth solution of \eqref{eq:ex} in
$\bar{B}_{R}(i_{0})$ for $\bar{T}$ sufficiently small and all $x\in S$.
In this case, we do not have constraints on $d$, $L$.
\end{thm}

\subsubsection*{The case of analytic initial conditions}

If all the $y_{0j}$, $j=\gamma,\ldots,d-1$, are analytic functions,
then for some $C_{j}\in\R_{>0}$, we have
\[
\|y_{0j}\|_{k+(n+1)L}\le C_{j}^{k+(n+1)L}\cdot\left(k+(n+1)L\right)!\quad\forall j=\gamma,\ldots,d-1\,\forall k\in\N.
\]
Using Stirling's approximation, we have $\left(k+(n+1)L\right)!\sim(nL)^{k+L}(nL)!$
and hence the following
\begin{thm}
\label{thm:exAnalytic}If $d\ge L$, then there exists a smooth solution
of \eqref{eq:ex} in $\bar{B}_{R}(i_{0})$ with analytic initial conditions
$y_{0j}$ if $j=\gamma,\ldots,d-1$ and \emph{arbitrary smooth} $y_{0j}$
if $j=0,\ldots,\gamma-1$, for $\bar{T}$ sufficiently small and all
$x\in S$.
\end{thm}

\noindent Note explicitly that already this theorem yields more general
results with respect to the classical Cauchy-Kowalevski theorem because
both the matrix coefficient $p(t)$ in \eqref{eq:ex} and the initial
conditions $y_{0j}$ for $j=0,\ldots,\gamma-1$ can be arbitrary smooth
functions.

\subsubsection*{A case of non-analytic initial conditions}

Now, let us assume that our initial conditions which are not analytic
satisfy
\begin{equation}
\|y_{0j}\|_{k+(n+1)L}\sim(nL)^{\sigma_{j}nL},\sigma_{j}>0\quad\forall j=\gamma,\ldots,d-1\,\forall k\in\N.\label{eq:nonAnalytic}
\end{equation}
Note that $\lim_{n\to+\infty}\frac{C^{n}n^{A}n!}{n^{\sigma_{j}n}}=0$,
so that each function $y_{0j}$ satisfying \eqref{eq:nonAnalytic}
cannot be an analytic function. We have
\[
\frac{\|y_{0j}\|_{k+(n+1)L}}{(nd)!}\sim\frac{e^{nd}}{\sqrt{2\pi nd}}\left(\frac{L}{d}\right)^{\sigma_{j}nL}\frac{1}{(nd)^{n(d-\sigma_{j}L)}}.
\]
We therefore have the following
\begin{thm}
\label{thm:exNonAnalytic}If the initial conditions $y_{0j}$, $j=0,\ldots,\gamma-1$,
are arbitrary smooth functions, whereas $y_{0j}$ for $j=\gamma,\ldots,d-1$
are analytic or they satisfy \eqref{eq:nonAnalytic}, and if in the
latter case we have $d>\sigma_{j}L$, then there exists a smooth solution
of \eqref{eq:ex} in $\bar{B}_{R}(i_{0})$ for $\bar{T}$ sufficiently
small and all $x\in S$.
\end{thm}

\noindent Taking $n\to+\infty$ in \eqref{eq:P^n}, we also obtain
the following generalization of Example \ref{exa:Aleksandr}:
\begin{cor}
\noindent \label{cor:formula}In the assumptions of each one of Thm.~\ref{thm:exp},
\ref{thm:exAnalytic}, \ref{thm:exNonAnalytic}, the solution $y$
of Picard-Lindelöf iterations is given by the formula:
\begin{equation}
y(t,x)=\sum_{j=0}^{\gamma-1}\frac{y_{0j}(x)}{j!}(t-t_{0})^{j}+\sum_{h=0}^{+\infty}\sum_{j=\gamma}^{d-1}\frac{\partial_{x}^{h\mu}y_{0j}(x)}{(j-\gamma)!}\mu_{j-\gamma,h}(t)+\sum_{h=0}^{+\infty}\eta_{h}(t,x).\label{eq:formula}
\end{equation}
In particular, if the functions $p$ and $q$ are constant, then
\begin{align*}
y(t,x) & =\sum_{j=0}^{\gamma-1}\frac{y_{0j}(x)}{j!}(t-t_{0})^{j}+\sum_{h=0}^{+\infty}\sum_{j=\gamma}^{d-1}\frac{\partial_{x}^{h\mu}y_{0j}(x)}{(j-\gamma)!}p^{h}\frac{(t-t_{0})^{h(d-\gamma)+j}}{\left[(d-\gamma)!\right]^{h}}+\\
 & \phantom{=}+q\sum_{h=0}^{+\infty}p^{h}\frac{(t-t_{0})^{(h+1)(d-\gamma)}}{\left[(h+1)(d-\gamma)\right]!}.
\end{align*}
Let $y(t,x;\eps)$ be the solution defined by \eqref{eq:formula}
corresponding to initial conditions $y_{0j}(x;\eps)$, where $\eps\in(-1,1)$.
If we can exchange $\lim_{\eps\to0}$ and $\sum_{h=0}^{+\infty}$,
e.g.~if the sequence of derivatives $\left(\partial_{x}^{h\mu}y_{0j}(x;\eps)\right)_{h\in\N}$
pointwise converges in a dominated way as $h\to+\infty$, i.e.~for
all $h\in\N$, $x\in S$, $j=0,\ldots,d-1$, and $\eps\in(-1,1)$
we have
\begin{gather*}
\exists\,\lim_{\eps\to0}\partial_{x}^{h\mu}y_{0j}(x;\eps)=\partial_{x}^{h\mu}y_{0j}(x;0)\\
|\partial_{x}^{h\mu}y_{0j}(x;\eps)|\le g_{h}(x;\eps)\\
\sum_{h=0}^{+\infty}g_{h}(x;\eps)<+\infty,
\end{gather*}
then $\lim_{\eps\to0}y(t,x;\eps)=y(t,x;0)$.
\end{cor}

Note that our estimates above of the Weissinger condition and the
radii $r_{k}$, allow us also to state that conditions \eqref{eq:WCor}
and \eqref{eq:iterationsSolEq} of Cor.~\ref{cor:solEqLoD} hold.
Moreover, the proof of PL Thm.~\ref{thm:PLpde} shows that also \eqref{eq:losCor}
holds. Therefore, to solve \eqref{eq:ex} in the space $X=\bar{B}_{R}(i_{0})$
we can also apply Cor.~\ref{cor:solEqLoD}, as we stated above in
Sec.~\ref{sec:equationsLoD}.

We close this section by noting that for the PDE
\begin{equation}
\partial_{t}^{d}y(t,x)=y(t,x)\cdot\partial_{x}^{\mu}y(t,x)\label{eq:genBurges}
\end{equation}
with $|\mu|=L$, we can use ideas similar to those of Thm.~\ref{thm:Lipschitz for PDE}
to show that setting
\begin{align*}
\bar{C}_{k+L} & :=\bigcup_{|\alpha|\le k+L}\partial_{x}^{\alpha}i_{0}(T\times S)\Subset\R^{m}\\
\|i_{0}(T\times S)\|_{k+L} & :=d(\bar{C}_{k+L},0)\\
\Lambda_{k} & :=2^{k}\left(r_{k+L}+\|i_{0}(T\times S)\|_{k+L}\right),
\end{align*}
then \eqref{eq:genBurges} has $(\Lambda_{k})_{k\in\N}$ as Lipschitz
constants with $L$ loss of derivatives. Thereby, $\bar{\Lambda}_{kn}=\frac{T^{nd}}{(nd)!}2^{n}\prod_{j=0}^{n-1}\left(r_{k+jL}+\|i_{0}(T\times S)\|_{k+jL}\right)$.
However, in the case $L=1$ and $y_{0j}$ satisfying \eqref{eq:nonAnalytic}
with $\sigma_{j}=1$, we get $\bar{\Lambda}_{kn}=\frac{T^{nd}}{(nd)!}2^{n}H(n-1)$,
where $H(n-1)=\prod_{j=0}^{n-1}j^{j}$ is the hyperfactorial function.
Since $H(n)=O(n^{n^{2}/2})$, Weissinger condition never holds. This
clearly left open the possibility of better estimates of different
Lipschitz factors.

\section{\label{sec:aboutLOD}Some remarks about the loss of derivatives condition}

As we discussed in the previous sections, Def.~\ref{def:contractionLoD}
of contraction with LOD is at the core of our version of the BFPT,
i.e.~Thm.~\ref{thm:BFPTLoss}. We start this section with a discussion
of this notion of contraction.
\begin{defn}
\label{def:L_P}We call \textit{minimal loss for $P$ from $y_{0}$
}the quantity
\[
L_{P}\left(y_{0}\right):=\min\left\{ L\in\mathbb{N}\mid P\in\LoD\left(X,L,y_{0}\right)\right\} .
\]
\end{defn}

\begin{lem}
\label{lem:fix point entails LoD 0}Let $\ensuremath{\left(\mathcal{F},\left(\Vert-\Vert_{k}\right)_{k\in\N}\right)}$
be a Fréchet space, $X$ be a closed subset of $\mathcal{F}$, $y_{0}\in X$
and $P:X\ra X$ be a continuous map. Assume that $\left\Vert -\right\Vert _{0}$
(hence, $\left\Vert -\right\Vert _{k}$ for every $k\in\N$) is a
norm. If there exists $N\in\N_{>0}$ such that $P^{N}\left(y_{0}\right)$
is a fixed point of $P$, then $P\in\LoD\left(X,0,y_{0}\right)$,
i.e.~\emph{$P$ is a contraction with }$0$ \emph{loss of derivatives
starting from }$y_{0}$.
\end{lem}

\begin{proof}
Let $N$ be the smallest number such that $P^{N}\left(y_{0}\right)$
is a fixed point of $P$.

If $N=1$, $P\left(y_{0}\right)=y_{0}$ hence $\left\Vert P^{n+1}\left(y_{0}\right)-P^{n}\left(y_{0}\right)\right\Vert _{k}=0$
for every $k$, $n\in\N$, so our claim follows just by setting each
$\alpha_{kn}:=0$.

If $N>1$, $\left\Vert P\left(y_{0}\right)-y_{0}\right\Vert _{0}\neq0$
since $\left\Vert -\right\Vert _{0}$ is a norm, therefore $\left\Vert P\left(y_{0}\right)-y_{0}\right\Vert _{k}\neq0$
for every $k\in\N$, as the norms $\left\Vert -\right\Vert _{k}$
are increasing. For every $k$, $n\in\N$ set 
\begin{equation}
\alpha_{kn}:=\begin{cases}
\frac{\left\Vert P^{n+1}\left(y_{0}\right)-P^{n}\left(y_{0}\right)\right\Vert _{k}}{\left\Vert P\left(y_{0}\right)-y_{0}\right\Vert _{k+nL}}, & \text{if}\,n<N;\\
\frac{1}{n^{2}\left\Vert P\left(y_{0}\right)-y_{0}\right\Vert _{k+nL}}, & \text{otherwise.}
\end{cases}\label{eq:alphaForLeq0}
\end{equation}

\noindent With this choice, Def.~\ref{def:contractionLoD}.\eqref{enu:LoD-contr}
and Def.~\ref{def:contractionLoD}.\eqref{enu:LoD-W} are easily
verified because $\alpha_{kn}\Vert P(y_{0})-y_{0}\Vert_{k+nL}=\frac{1}{n^{2}}\ge\Vert P^{n+1}(y_{0})-P^{n}(y_{0})\Vert_{k}=\Vert y_{0}-y_{0}\Vert_{k}$
if $n\ge N$, hence $P\in\LoD\left(X,0,y_{0}\right)$.
\end{proof}
The following is a rather surprising fact that holds, e.g., in $\mathcal{C}_{t}^{0}\mathcal{C}_{x}^{\infty}\left(T\times S,\R^{m}\right)$.
\begin{thm}
\label{thm:L=00003D0}Let $\ensuremath{\left(\mathcal{F},\left(\Vert-\Vert_{k}\right)_{k\in\N}\right)}$
be a Fréchet space, let $X$ be a closed subset of $\mathcal{F}$,
let $y_{0}\in X$ and let $P:X\ra X$. Assume that $\left\Vert -\right\Vert _{0}$
(hence, $\left\Vert -\right\Vert _{k}$ for every $k\in\N$) is a
norm. The following properties are equivalent:
\begin{enumerate}
\item \label{enu:eq_cond_LoD1}There exists $L\in\N$ such that $P\in\LoD\left(X,L,y_{0}\right)$;
\item \label{enu:eq_cond_LoD2}$P$ is continuous and for all $k\in\N$
\begin{equation}
\sum_{n=0}^{\infty}\left\Vert P^{n+1}\left(y_{0}\right)-P^{n}\left(y_{0}\right)\right\Vert _{k}<+\infty;\tag{W'}\label{eq:W'}
\end{equation}
\item \label{enu:eq_cond_LoD3}$P\in\LoD\left(X,0,y_{0}\right)$.
\end{enumerate}
\end{thm}

\begin{proof}
\ref{enu:eq_cond_LoD1}$\Rightarrow$\ref{enu:eq_cond_LoD2}: If there
exist $k$, $N\in\N$ such that $\left\Vert P^{N+1}\left(y_{0}\right)-P^{N}\left(y_{0}\right)\right\Vert _{k}=0$,
then $P^{N}\left(y_{0}\right)$ is a fixed point of $P$ as $\left\Vert -\right\Vert _{k}$
is a norm. But then, for every $m\geq N$ and for every $k\in\N$
we have $\left\Vert P^{m+1}\left(y_{0}\right)-P^{m}\left(y_{0}\right)\right\Vert _{k}=0$,
and hence $\sum_{n=0}^{\infty}\left\Vert P^{n+1}\left(y_{0}\right)-P^{n}\left(y_{0}\right)\right\Vert _{k}=\sum_{n=0}^{N-1}\left\Vert P^{n+1}\left(y_{0}\right)-P^{n}\left(y_{0}\right)\right\Vert _{k}<+\infty$.
Otherwise, $||P(y_{0})-y_{0}||_{k+nL}\neq0$ for every $k$, $N\in\N$,
hence by Def.~\ref{def:contractionLoD}.\ref{enu:LoD-contr} we get
that $\alpha_{kn}\geq\frac{||P^{n+1}(y_{0})-P^{n}(y_{0})||_{k}}{||P(y_{0})-y_{0}||_{k+nL}}\in\R_{>0}$
which, substituted in Def.~\ref{def:contractionLoD}.\ref{enu:LoD-W},
gives \eqref{eq:W'}.

\ref{enu:eq_cond_LoD2}$\Rightarrow$\ref{enu:eq_cond_LoD3}: If there
exist $k$, $N\in\N$ such that $\left\Vert P^{N+1}\left(y_{0}\right)-P^{N}\left(y_{0}\right)\right\Vert _{k}=0$,
then $P^{N}\left(y_{0}\right)$ is a fixed point of $P$ as $\left\Vert -\right\Vert _{k}$
is a norm, and we conclude by Lemma \ref{lem:fix point entails LoD 0}.
Otherwise, in particular $\left\Vert P\left(y_{0}\right)-y_{0}\right\Vert _{k}\neq0$
for every $k\in\N$. For every $k$, $N\in\N$, we set $\alpha_{kn}:=\frac{\left\Vert P^{n+1}\left(y_{0}\right)-P^{n}\left(y_{0}\right)\right\Vert _{k}}{\left\Vert P\left(y_{0}\right)-y_{0}\right\Vert _{k+nL}}$.
Then Def.~\ref{def:contractionLoD}.\ref{enu:LoD-contr} holds trivially,
and Def.~\ref{def:contractionLoD}.\ref{enu:LoD-W} holds as, by
construction
\[
\sum_{n=0}^{\infty}\alpha_{kn}\left\Vert P\left(y_{0}\right)-y_{0}\right\Vert _{k}=\sum_{n=0}^{\infty}\left\Vert P^{n+1}\left(y_{0}\right)-P^{n}\left(y_{0}\right)\right\Vert _{k}<+\infty
\]
by assumption.

\ref{enu:eq_cond_LoD3}$\Rightarrow$\ref{enu:eq_cond_LoD1}: This
is trivial.
\end{proof}
In particular, this result shows that, if $\Vert-\Vert_{k}$ are norms,
$L_{P}\left(y_{0}\right)=0$ whenever $P\in\LoD\left(X,L,y_{0}\right)$
for some $L$. Note that, in general, this does not entail the uniqueness
of the fixed point of $P$, since such uniqueness would require a
much stronger condition on $P$ than Def.~\ref{def:contractionLoD}.\ref{enu:LoD-contr}
or condition \eqref{eq:W'}, see e.g.~Lem.~\ref{lem:classicalContr}.

We also note that condition \eqref{eq:W'} implies that $\left(P^{n}\left(y_{0}\right)\right)_{n\in\N}$
is a Cauchy sequence as we did in \eqref{eq:cauchyP}, and this, together
with the continuity of $P$, yields that $\overline{y}:=\lim_{n\ra+\infty}P^{n}\left(y_{0}\right)$
is a fixed point of $P$ by Lem.~\ref{lem:Banach_banale}.

On the other hand, the previous Thm.~\ref{thm:L=00003D0} \emph{does
not} imply that we can take $L=0$ in the PLT Thm.~\ref{thm:PLpde},
because the assumption that the right hand side $G$ of the PDE is
Lipschitz on $\bar{B}_{R}(y_{0})$ with loss of derivatives $L=0$
in general is \emph{not} satisfied. In other words: The natural loss
of derivatives $L>0$ corresponds to the maximum order of derivatives
in $x$ appearing in the PDE \eqref{eq:PDE}, and the natural Lipschitz
constants $\alpha_{kn}$ are derived in the proof of Thm.~\ref{thm:PLpde},
e.g.~using the Lipschitz factors $\left(\Lambda_{k}\right)_{k\in\N}$
for the right hand side of the PDE derived from Thm.~\ref{thm:Lipschitz for PDE}.
Using these natural constants, Weissinger condition \ref{enu:W} is
easier to estimate than condition \eqref{eq:W'} or the use of \eqref{eq:alphaForLeq0}.

\section{Conclusions}

Starting from the classical Kowalevski counter-example for the heat
equation or Hadamard's results on the Cauchy problem for the Laplace
equation, one can think that a PDE links in a given relation $\partial_{t}y$
and $\partial_{x}y$ and hence it necessarily forces the solution,
in general, in a space of functions whose derivatives growth in a
restricted way, these constraints being related to the PDE itself.
This implies that the initial conditions cannot be freely chosen but
must be taken into another constrained space. We could say that we
do not have to find a suitable space of generalized solutions for
our PDE, but conditions stating when it has a solution or not; only
at the philosophical level, this is similar to the point of view of
nonlinear differential Galois theory, see e.g.~\cite{Mal}, or the
formal theory of differential equations, see e.g.~\cite{Sei}.

The PLT we proved in this paper goes exactly in this direction, by
showing that the existence of a solution (by Picard-Lindelöf iterations)
depends on the initial conditions we start with: Def.~\ref{def:contractionLoD}
of contraction with loss of derivatives, the closure with respect
to iterations \ref{enu:LOD-iterates}, Weissinger condition \eqref{eq:W},
the definition of the radii \eqref{eq:P^n}, all go in this direction.
Examples considered in Sec.~\ref{sec:Examples} show a first link
between the syntax of the PDE (in the term $(nd)!$) and the order
of growth of the derivatives of the initial conditions $\|y_{0j}\|_{k+nL}$.
On the other hand, exactly as up to fourth order algebraic equations
are solvable in radicals, if the order $d=1$ the method of characteristics
allows one to solve a large class of PDE for \emph{any} initial condition.

It is now natural to ask for a generalization to more singular normal
(nonlinear) PDE, e.g.~where the right hand side $F$ or some of the
initial conditions $y_{0j}$ are some kind of generalized functions.
In order to get this generalization by following the ideas of the
present work, we would need a space of generalized functions which
is closed with respect to composition and with a complete topology
generated by norms; this space must clearly be non-trivial, e.g.~containing
all Sobolev-Schwartz distributions. In our opinion, this target can
be fully accomplished in a beautiful and simple setting by considering
the Grothendieck topos of non-Archimedean generalized smooth function,
see e.g.~\cite{GiKuVe21,LuGi22,GiLu22}. We plan to realize this
goal in future works.

\end{document}